\documentclass[12pt]{amsart}

    \usepackage{amsmath,amsthm, amsfonts,amssymb,thm-restate}  
    \usepackage{float}
    \usepackage{graphicx}
    \usepackage[linktocpage,bookmarksopen,bookmarksnumbered,
                pdftitle={Monodromy Action on Tunnels in Fiber Surfaces},
                pdfauthor={Jessica E. Banks, Matt Rathbun},%
                pdfkeywords={tunnel, fiber, monodromy, clean}]{hyperref}
    
    \newtheorem{Theorem}{Theorem}[section]
    \newtheorem{Proposition}[Theorem]{Proposition}
    \newtheorem{Lemma}[Theorem]{Lemma}
    \newtheorem{Corollary}[Theorem]{Corollary}
    
    \theoremstyle{definition}
    \newtheorem{Definition}[Theorem]{Definition}
    
    \newtheorem{Note}[Theorem]{Note}
        
    \newcommand{\set}[1]{\{ #1 \} }
    \newcommand{\bd}{\partial}
    \newcommand{\rmv}{\smallsetminus}
    \DeclareMathOperator{\Id}{Id}
    \DeclareMathOperator{\Int}{int}
    \DeclareMathOperator{\dist}{d}

\begin{document}

\title[Monodromy of tunnels in fiber surfaces]{Monodromy action on unknotting tunnels in fiber surfaces}
\author{Jessica E. Banks}
\author{Matt Rathbun}
\date{June 3, 2014}
\begin{abstract} In \cite{RatTOFL}, the second author showed that the tunnel of a tunnel number one, fibered link can be isotoped to lie as a properly embedded arc in the fiber surface of the link. In this paper, we analyze how the arc behaves under the monodromy action, and show that the tunnel arc is nearly clean, with the possible exception of twisting around the boundary of the fiber.
\end{abstract}
\maketitle


\section{Introduction}

The Berge Conjecture is a long-standing conjecture that attempts to classify all knots in $S^3$ that admit Dehn surgeries resulting in a lens space. Such a classification is foundational to understanding Dehn surgery on 3-manifolds, and has been a motivating topic of research in low dimensional topology for decades. All knots conjectured by Berge to admit such surgeries are both \emph{tunnel number one}, and \emph{fibered}. Conversely, Yi Ni proved that if a knot admits such a surgery, then it must be fibered, \cite{NiKFHDFK}. In light of this, we aim to understand tunnel number one, fibered knots and links.

In Section \ref{section:definitions}, we will define three well-understood operations on fibered links: Stallings twists, Hopf plumbing, and its inverse Hopf de-plumbing. All three of these operations can be characterized by arcs that are \emph{clean}, i.e.\ disjoint from their images under the monodromy map (except at their endpoints). 

Our goal in this paper is to understand how the monodromy acts on tunnels sitting as arcs in the fiber. We show that such tunnels sit \emph{nearly} cleanly in the fiber. We prove the following theorem:

\begin{restatable}{Theorem}{TheoremTunnelAlmostClean}\label{Theorem:TunnelAlmostClean}
Suppose $K$ is a tunnel number one, fibered link, with fiber $F$, monodromy $h$, and tunnel $\tau \subset F$. Then there exists an arc $\beta$, freely isotopic to $h(\tau)$ rel $\bd F$, so that $\Int(\tau) \cap \Int(\beta) = \emptyset$. In particular, for a suitable choice of monodromy map $h$ within its isotopy class, there exists a regular neighborhood of $\bd F$ outside of which $\tau$ and $h(\tau)$ do not intersect.
\end{restatable}

Johnson \cite{JohSBGTHS} investigated closed surface bundles with genus two Heegaard splittings. Johnson's work gives a description of the monodromy of a fibred tunnel number one knot, but it does not tell us about the case of a two-component link. Sakuma proved that tunnels are, in fact, clean for once-punctured torus bundles, \cite{SakUTCDPTBOC}. According to a survey article by Sakuma, \cite{SakTGAUT}, this result was proven for arbitrary punctured surface bundles by Kobayashi and independently by Johannson. However, both references are talks, and the result cannot hold in general for punctured surface bundles, as we will discuss in Section \ref{section:dehntwists}.

This paper is organized as follows: Section \ref{section:definitions} details definitions and background necessary for the statement and proof of the main theorem, found in Section \ref{section:analyzing}. Section \ref{section:dehntwists} discusses limitations of the theorem owing to difficulties associated with fractional Dehn twists around the boundary of the fiber surface. And finally, Section \ref{section:cusps} provides an application to bounding the cusp area for hyperbolic, fibered knots. 

The authors would like to give special thanks to Ken Baker, Kai Ishihara, and Dave Futer, who helped to improve this paper substantially.


\section{Definitions and Background}\label{section:definitions}

\begin{Definition} A link $K$ is said to have \emph{tunnel number one} if there exists an arc $\tau$ (the \emph{tunnel}) properly embedded in the link exterior so that the exterior of $K \cup \tau$ is a (genus two) handlebody.
\end{Definition}

A tunnel number one link can therefore have at most two link components, and in this case, the tunnel must have one endpoint on each component.
Tunnel number one knots and links have been studied in great depth (see, for example, \cite{SchTNOKSPC},  \cite{GoReTDTNOKL}, \cite{MoSaYoITNOK}, \cite{HiTeToTNOKHPP}). Recently, Cho and McCullough have given a bijective correspondence between tunnel number one knots (with their tunnels) and a subset of vertices of a certain tree related to a subcomplex of non-separating disks in a genus two handlebody \cite{ChoMcCTKT}. They are further able to parameterize all tunnel number one knots by a sequence of `cabling' operations (see \cite{ChoMcCCSTTK} and \cite{ChoMcCCKTUGS}). While the cabling operation is a very natural way of describing and modifying knots, it is generally not clear how properties of the exterior change. 

\begin{Definition}
Let $K \subset S^3$ be a link. A \emph{Seifert surface} for $K$ is a compact, orientable surface $F$, with no closed components, embedded in $S^3$ such that $\bd F = K$.
\end{Definition}

\begin{Definition}
A map $f \colon E \to B$ is a \emph{fibration with fiber $F$} if, for every point $p \in B$, there is a neighborhood $U$ of $p$ and a homeomorphism $h \colon f^{-1}(U) \to U \times F$ such that $f = \pi_1 \circ h$, where $\pi_1 \colon U \times F \to U$ is projection to the first factor. The space $E$ is called the \emph{total space}, and $B$ is called the \emph{base space}. Each set $f^{-1}(b)$ is called a \emph{fiber}, and is homeomorphic to $F$.
\end{Definition}

\begin{Definition}
A link $K \subset S^3$ is said to be \emph{fibered} if there is a fibration of $S^3 \setminus n(K)$ over $S^1$, and the fibration is well-behaved near $K$. That is, each link component $K'$ of $K$ has a neighborhood $S^1 \times D^2$, with $K' = S^1 \times \set{0}$ such that $f|_{S^1 \times (D^2 \setminus \set{0})}$ is given by $(x, y) \to \frac{y}{|y|}$.
\end{Definition}

\begin{Definition}
Let $K$ be a fibered link in $S^3$ with fiber $F$. Then $S^3 \setminus n(K)$ can be obtained from $F \times I$ by the identification $(x, 0) \sim (h(x), 1)$ for $x \in F$, where $h \colon F \to F$ is an orientation-preserving homeomorphism that is the identity on $\bd F$. We call $h$ a \emph{monodromy map}.
\end{Definition}

Note that $h$ is well-defined up to conjugation by an element of the mapping class group of $F$, and a choice of marking on $\bd n(K)$ to distinguish the meridian(s) of $K$. In particular, if $\widetilde{h}$ differs from $h$ by a product of Dehn twists about a curve parallel to $\bd F$, then $(F \times I) / h \cong (F \times I) / \widetilde{h}$. However, Dehn filling the torus boundary along the curve defined by $\set{pt.} \times I$ in each case may result in different closed $3$-manifolds, related by $\pm (\frac{1}{n})$--surgery. 

Fibered knots, too, have been studied in great depth (see, for example, \cite{BirRTFKTMM}, \cite{HarHTCAFKL}, \cite{NakCGFK}, \cite{BaJoKlTNOGOFK}). Stallings described a pair of operations on fibered links that result in new fibered links, what are now called the \emph{Murasugi sum} and \emph{Stallings twists}, \cite{StaCFKL}. Harer then showed that twists and a certain type of Murasugi sum called \emph{Hopf plumbing} (and its inverse, \emph{Hopf de-plumbing}) were sufficient to transform any fibered link into any other fibered link, \cite{HarHTCAFKL}. (In fact, recent work of Giroux and Goodman showed that Stallings twists are not necessary, \cite{GirGooOSEOB3M}.) 

\begin{Definition}
Let $F$ be a Seifert surface for a link $L$.
Let $\alpha$ be an arc properly embedded in $F$.
\emph{Hopf plumbing along} $\alpha$ is a change in the surface $F$ within a neighborhood of the arc $\alpha$, as shown in Figure \ref{Figure:plumbing}. That is, a disk is attached to $F$ along two sub-arcs of its boundary. The positioning of the disk is defined by $\alpha$, and the disk contains a full twist relative to $F$. Given $F$ and $\alpha$ there are two ways to perform Hopf plumbing, distinguished by the handedness of this twisting.
The result is a new surface $F'$ and a new link $K' = \bd F'$.
\begin{figure}[hbtp]
\begin{center}
\includegraphics[width=6cm]{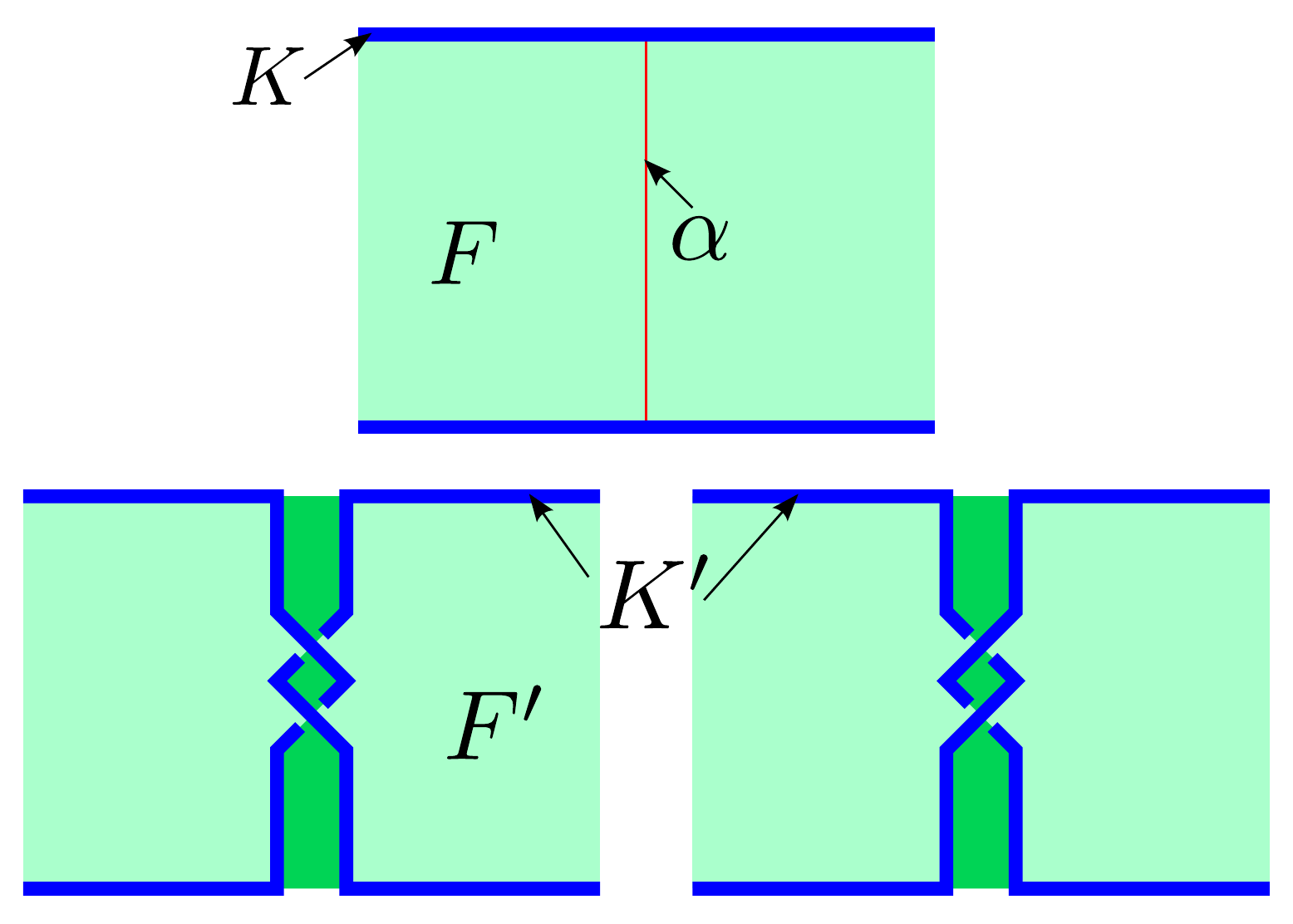}
\caption{Hopf plumbing is a change in a surface $F$ in the neighborhood of an arc $\alpha$.}
\label{Figure:plumbing}
\end{center}
\end{figure}
\end{Definition}

\begin{Note} Suppose $F$ is a Seifert surface for the link $\bd F$, and Hopf plumbing results in a Seifert surface $F'$ for the link $\bd F'$. Then $F$ is a fiber surface if and only if $F'$ is a fiber surface (see \cite{GabMSNGOII}).
\end{Note}

\begin{Note} De-plumbing a Hopf band corresponds exactly to cutting the fiber surface along an arc that is clean and alternating with respect to the monodromy. (Here, \emph{alternating} means that if $\alpha \times [0, 1]$ is a small product neighborhood of the arc $\alpha$ in $F$, then the image of $\alpha$ intersects both $\alpha \times \set{0}$ and $\alpha \times \set{1}$ in a neighborhood of the endpoints. Otherwise, say that $\alpha$ is \emph{non-alternating}.) This is implicit in work of Gabai (\cite{GabDFLS3}), and attributed to Sakuma (\cite{SakMDCSSUO}). For a proof, see Coward--Lackenby \cite{CoLaUGOK}.
\end{Note}

\begin{Definition}
Let $c$ be a simple closed curve, embedded and essential in a fiber surface $F$ in a manifold $M$. Call $c$ a \emph{twisting curve} if $c$ bounds an embedded disk in $M$, and the framing of $c$ from the disk agrees with the framing of $c$ from $F$. In this case, performing a $\pm 1$-Dehn surgery on $c$ is called a \emph{Stallings twist}. See Figure \ref{Figure:stallings}.
\begin{figure}[hbtp]
\begin{center}
\includegraphics[width=11.5cm]{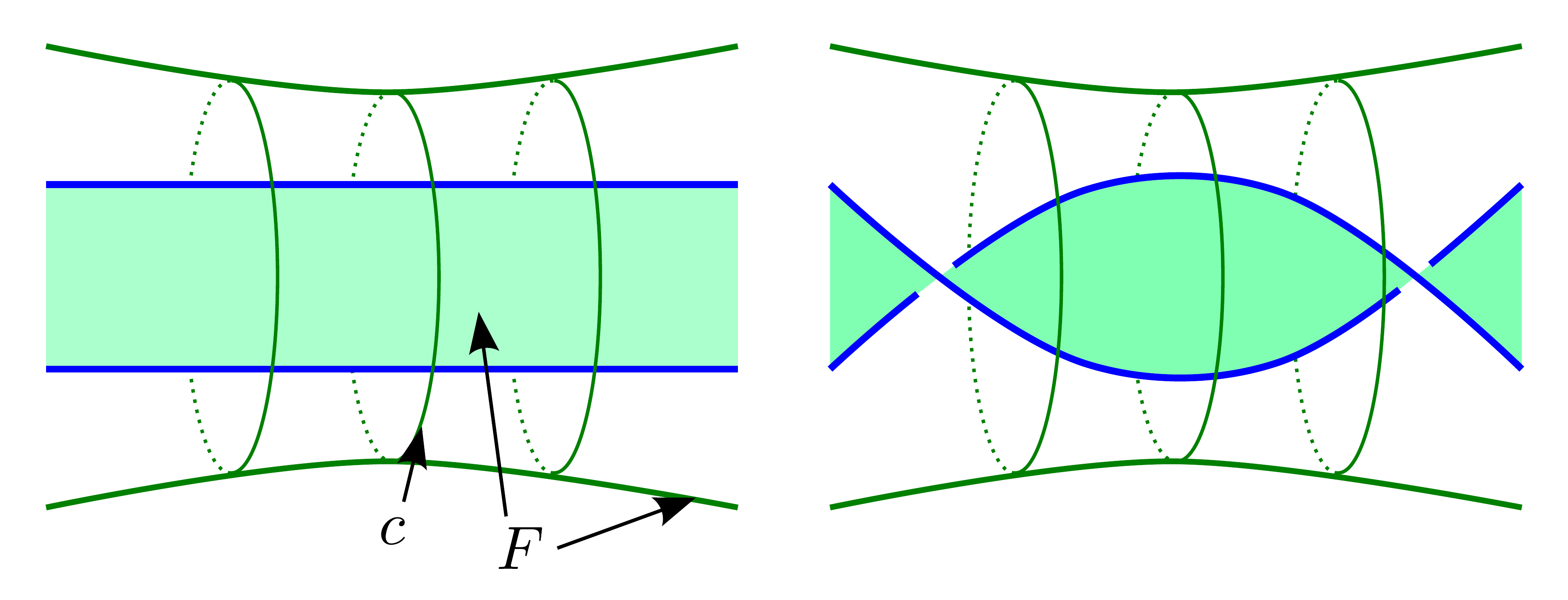}
\caption{A Stallings twist results from a $\pm 1$-Dehn surgery on an unknotted curve in the fiber surface.}
\label{Figure:stallings}
\end{center}
\end{figure}
\end{Definition}

\begin{Note}
Stallings proved that the image of a fibered link under such a twist was another fibered link, with fiber surface homeomorphic to the original fiber surfrace \cite{StaCFKL}.
\end{Note}

\begin{Note} Yamamoto (\cite{YamSTRPDHB}) proved that the existence of a Stallings twist of a certain type (type $(0, 1)$) corresponds exactly to twisting \emph{around} an arc that is clean and non-alternating with respect to the monodromy (i.e.\ that the disk bounded by the twisting curve intersects the fiber surface exactly in such an arc).
\end{Note}

These operations happen to interact very nicely with respect to unknotting tunnels. In \cite{RatTOFL}, the second author showed that in a tunnel number one, fibered link exterior, an unknotting tunnel can be isotoped to lie in a fiber surface. If a Hopf plumbing is performed along a tunnel lying in the fiber surface, then the resulting link is fibered and is again tunnel number one. Conversely, if de-plumbing a Hopf band corresponds to cutting along a tunnel lying in the fiber surface, then the resulting link is fibered and tunnel number one. Similarly, when twisting \emph{around} a tunnel lying in a fiber surface, the resulting link is fibered and still tunnel number one. 

Hence, a correspondence between tunnel arcs and the pre-conditions to perform Hopf (de-)plumbing or Stallings twists would provide a recipe for constructing tunnel number one, fibered knots. 

    
\section{Analyzing a Tunnel in a Fiber}\label{section:analyzing}

By \cite{RatTOFL}, we know that an unknotting tunnel for a fibered, tunnel number one link can be isotoped to lie in a fiber. Henceforth in this paper, we consider unknotting tunnels to be properly embedded arcs in a fiber surface. We now investigate \emph{how} such a tunnel can lie in a fiber.
The main aim of this section is to prove Theorem \ref{Theorem:TunnelAlmostClean}, but the majority of the work lies in establishing the following proposition.

\begin{Proposition}\label{Proposition:TunnelCleanAfterIsotopy}
Let $F$ be a compact, orientable surface with genus at least one and either one or two boundary components, and let $h \colon F \to F$ be a homeomorphism such that $h(x) = x$ for $x \in \bd F$.
Let $M = (F \times I)/h$, and denote by $F$ the surface $F \times \set{0}$ in $M$.
Assume $M$ is not a handlebody.
Let $\tau$ be an arc properly embedded in $F$ such that $M \rmv n(\tau)$ is a (genus two) handlebody, where $n(\tau)$ is a regular neighborhood of $\tau$ in $M$.
Then there is an arc that is freely isotopic in $F$ to $h(\tau)$ and is disjoint from $\tau$.
\end{Proposition}

\begin{proof}
As $M$ is not a handlebody but $M \rmv n(\tau)$ is, the arc $\tau$ must be essential in $F$.
Set $F' = F \rmv n(\tau)$.
In addition, in $F \times I$ let $\tau_1$ be the copy of $\tau$ in $F \times \set{1}$ and $\tau_0$ the copy in $F \times \set{0}$.
Observe that $\pi(\tau_0) = h(\pi(\tau_1))$, where $\pi \colon F \times I \to F$ is projection. As $h|_{\bd F} = \Id$, the endpoints of $\pi(\tau_0)$ and $\pi(\tau_1)$ coincide.
Recall that $F \times I$ is irreducible and $F \times \set{0,1}$ is incompressible in $F \times I$.

Let $A$ be the annulus $\bd n(\tau) \rmv \bd M$. Then $A$ is divided into two rectangles by $F$. Let $A_1$ be the rectangle incident to $F \times \set{1}$, and $A_0$ the rectangle incident to $F \times \set{0}$. 
We may think of $A_1$ as a neighborhood of $\tau_1$ contained in $F \times \set{1}$, and similarly for $A_0 \subset F \times \set{0}$.

The proof of Proposition \ref{Proposition:TunnelCleanAfterIsotopy} works by controlling certain disks within $M \rmv n(\tau)$, in particular how they relate to the annulus $A$. We now build up some language to describe these disks.

\subsection{Special Arcs}

Let $D$ be a disk properly embedded in $F \times I$ such that $\bd D$ is transverse to $\bd F \times \set{0,1}$.

\begin{Lemma}
\label{Lemma:DisksCannotMissF}
No essential disk in $F \times I$ can be disjoint from $F \times \set{i}$ for $i \in \set{0,1}$.
In particular, if $\bd D \cap (\bd F \times \set{0,1}) = \emptyset$ then $D$ is inessential in $F \times I$.
\end{Lemma}
\begin{proof}
Without loss of generality, suppose that $D$ is an essential disk in $F \times I$ that is disjoint from $F \times \set{1}$.
Then every arc in $\bd D \cap (\bd F \times I)$ is inessential in $\bd F \times I$. 
On the other hand, any simple closed curve in $\bd F \times I$ is either trivial or parallel to a component of $\bd F \times \set{0}$.
We may therefore isotope $\bd D$ into $F \times \set{0}$. 
This contradicts that $F \times \set{0,1}$ is incompressible in $F \times I$. Thus no such disk exists.
\end{proof}

\begin{Definition}
If $\bd D \cap (\bd F \times \set{0,1}) \neq \emptyset$ then the points of 
$\bd D\cap (\bd F \times \set{0,1})$ divide $\bd D$ into a finite set of sub-arcs of the following six possible types.
\begin{enumerate}
\item Sub-arcs in $F \times \set{0}$ parallel in $F$ to $\tau_0$; call these \emph{$\tau_0$--arcs}.
\item Sub-arcs in $F \times \set{1}$ parallel in $F$ to $\tau_1$; call these \emph{$\tau_1$--arcs}.
\item Sub-arcs in $\bd F \times I$; call these \emph{boundary arcs}.
\item Sub-arcs in $F \times \set{0}$ or $F \times \set{1}$ that are trivial in $F$; call these \emph{extra arcs}.
\item Sub-arcs in $F \times \set{i}$ for $i \in \set{0,1}$ that are essential in $F$ but are not $\tau_i$--arcs and are disjoint from $\tau_i$; call these \emph{special arcs}.
\item Sub-arcs in $F \times \set{i}$ for $i \in \set{0,1}$ that are essential in $F$, are not $\tau_i$--arcs, and intersect $\tau_i$; call these \emph{bad arcs}.
\end{enumerate}
        
For $i \in \set{0,1}$, label each sub-arc of $\bd D$ with $i$ if it is contained in $F \times \set{i}$.
\end{Definition}
        
None of the disks that will be of interest to us will have bad arcs, so we will mostly not consider them any further.

\begin{Definition}
An extra arc that is outermost in $F \times \set{i}$ can be isotoped off $F \times \set{i}$, along the subdisk it cuts off from $F \times \set{i}$, joining two sub-arcs on $\bd F \times I$ into a single boundary arc. Call this a \emph{tightening-move}. Notice that this does not affect the isotopy type of any essential arc in $F \times \set{0, 1}$, and has the effect of deleting an $i$--label from the labeling of $\bd D$.
\end{Definition}

If $|\bd F| = 2$, the following definition gives two isotopy classes of arcs in $F$ that will be of special interest to us. These arcs are boundary-parallel in $F'$, and have both endpoints on the same component of $\bd F$. The two isotopy classes are distinguished by which component of $\bd F$ contains the endpoints of the arc.

\begin{Definition}
Call a special arc a \emph{$\tau_2$--arc} if it is parallel in $F$ to the union of the two arcs in $\bd A_i \rmv \bd F$ and one of the two components of $\bd F \rmv A_i$. See Figure \ref{Figure:tau2arc}. Roughly speaking, it runs parallel to $\tau_i$, around $\bd F$ while avoiding $\tau_i$, and then back parallel to $\tau_i$.
\begin{figure}[hbtp]
\begin{center}
\includegraphics[height=2.75cm]{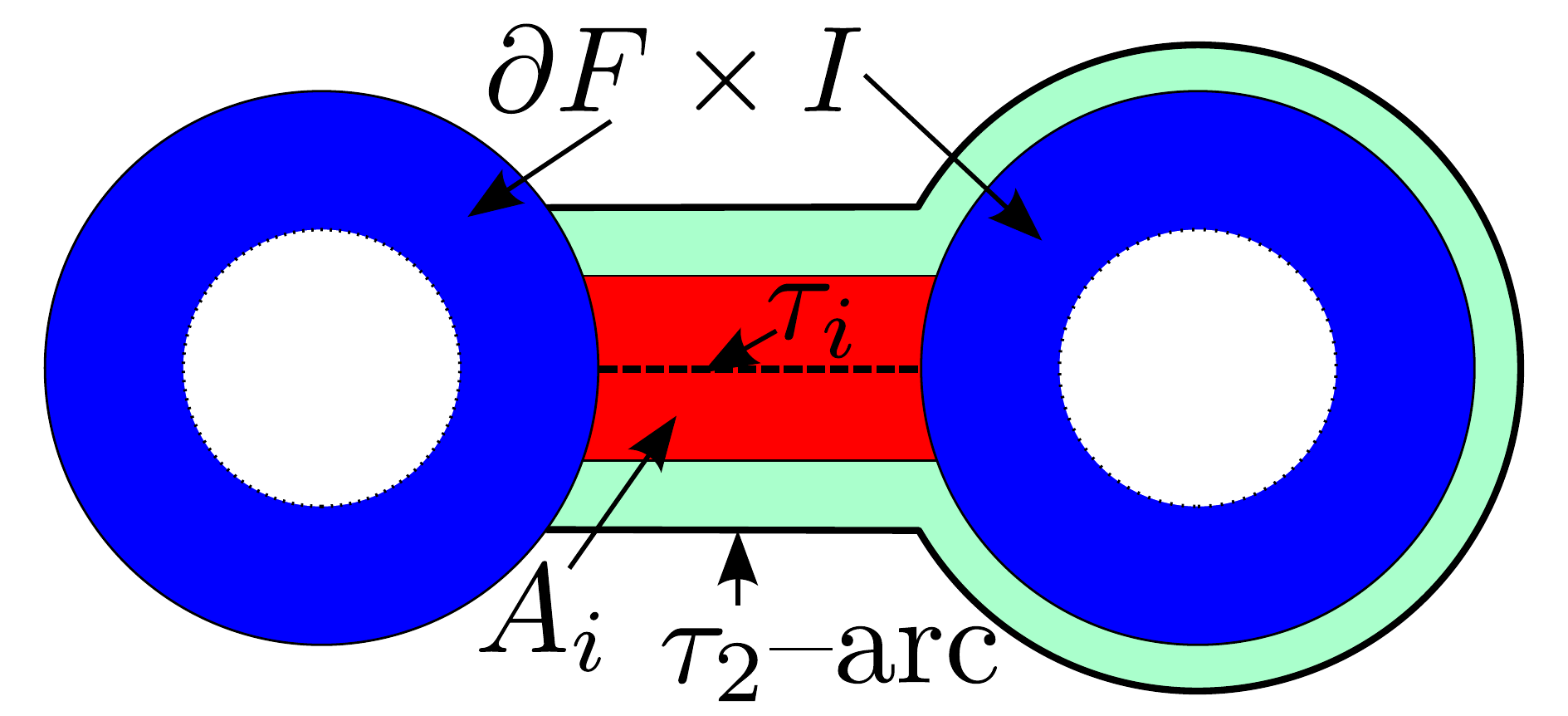}
\caption{A $\tau_2$--arc runs parallel to $\tau_i$, around $\bd F$, and back parallel to $\tau_i$.}
\label{Figure:tau2arc}
\end{center}
\end{figure}
\end{Definition}

Although not important for our purposes, it is interesting to note that, if a $\tau_2$--arc $\alpha$ exists in $\bd D$, pushing part of $\alpha$ across the disk of $F'$ cut off by $\alpha$ into the component of $\bd F \times I$ that does not contain the endpoints of $\alpha$ would change the special arc $\alpha$ into two $\tau_i$--arcs and one boundary arc.

The significance of $\tau_2$--arcs is their appearance in following lemma.

\begin{Lemma}\label{Lemma:OneSpecialArc}
If $\bd D$ contains exactly one special arc and no bad arcs then either $D$ is essential in $F \times I$, or $|\bd F|=2$ and the special arc is a $\tau_2$--arc. 
\end{Lemma}
\begin{proof}
Without loss of generality, assume the special arc $\alpha$ is labelled $1$. Perform as many tightening-moves 
 as possible to remove all extra arcs. This neither creates any new special or bad arcs, nor alters $\alpha$. 
Then $\bd D \cap (F \times \set{1})$ consists of $\alpha$ together with some number of $\tau_1$--arcs.

Suppose $D$ is boundary parallel in $F \times I$, and let $D'$ be the disk in $\bd(F \times I)$ to which $D$ is parallel. Consider the two components of $(F \times \set{1}) | \bd D$ adjacent to $\alpha$, one of which is a subsurface of $D'$. Call this $D''$. 
Note that $D''$ is planar, and all but one of the components of $\bd D''$ are contained in $\Int(D')$ and therefore
are components of $\bd F$. Any such components of $\bd F$ must also bound disks in $\bd(F \times I)$. 
Since there are no such components of $\bd F$, we see that $D''$ is a disk.
There can be at most two $\tau_1$--arcs in $\bd D''$, and exactly one copy of $\alpha$. 

If $\bd D''$ contained no $\tau_1$--arcs then $D''$ would provide an isotopy of $\alpha$ into $\bd F$, which is not possible. If $\bd D''$ contained exactly one $\tau_1$--arc then $D''$ would provide an isotopy of $\alpha$ into $\tau_1$, which is also impossible. Therefore $\bd D''$ contains two $\tau_1$--arcs. Notice that $\bd D'' \rmv \alpha$ is contained in $\bd F'$.
Suppose that $|\bd F|=1$. Then $\bd F'$ has two components, with one copy of $\tau$ in each. Therefore no such disk $D''$ can exist if $|\bd F|=1$. Hence $|\bd F|=2$. The disk $D''$ demonstrates that $\alpha$ is a $\tau_2$--arc.
\end{proof}

\subsection{Special Disks}
Lemma \ref{Lemma:OneSpecialArc} shows that disks whose boundary contains no bad arcs and only one special arc are important. This motivates the following definition.
        
\begin{Definition}
\label{Definition:SpecialDisk}
Given a disk $D$ properly embedded in $F \times I$ such that $\bd D$ is transverse to $\bd F \times \set{0,1}$,
say that $D$ is \emph{special} if it is essential in $F \times I$, and there are no bad arcs and at most one special arc in $\bd D$.
We will call $D$ a \emph{$0$-special} or \emph{$1$-special} disk depending on the label and location of the special arc if one exists. If there is no special arc, say the disk is $1$-special.
\end{Definition}

\begin{Lemma}\label{Lemma:SpecialDisksExist} 
There exist special disks in $(M \rmv n(\tau)) | F'$.
\end{Lemma}

\begin{proof} 
As $M \rmv n(\tau)$ is a genus two handlebody, we know that $\bd (M \rmv n(\tau))$ is compressible in $M \rmv n(\tau)$.
Let $D'$ be a compression disk such that $\bd D' \cap A$ consists of straight arcs, each essential in $A$ and running from one component of $A_i \cap \bd F$ to the other, and such that $|D' \cap F'|$ is minimal among such disks. 
Since $\bd M$ is incompressible in $M \rmv n(\tau)$, we know that $\bd D'$ runs across $A$ at least once.

If $D' \cap F' = \emptyset$ then $D'$ is a disk in $F \times I$ and $\bd D'$ contains no special or bad arcs. 
Note that $D'$ is essential in $F \times I$ since it is essential in $M \rmv n(\tau)$ and $F'$ is not a disk. Therefore $D'$ is a special disk.

If $D' \cap F' \neq \emptyset$, then notice that $D' \cap F'$ consists only of arcs, since circles of intersection innermost in $D'$ and essential in $F$ would give rise to compressions for $F$, and inessential ones could be removed to reduce $|D' \cap F'|$. 
Moreover, as $\tau$ is essential in $F$, the minimality of $|D' \cap F'|$ implies that every arc of $D' \cap F'$ is essential in $F$. Knowing this, the minimality of $|D' \cap F'|$ further implies that no arc of $D' \cap F'$ is isotopic to $\tau$ in $F$.

Consider an arc $\alpha$ of $D' \cap F'$ that is outermost in $D'$. Then $\alpha$ cuts off a subdisk $D$ from $D'$. 
Now view $D$ as a disk in $F \times I$.
Without loss of generality, assume $\alpha$ is labeled $1$. Note also that $\alpha$ is a special arc.
Because $\bd D$ contains exactly one special arc and no bad arcs, by Lemma \ref{Lemma:OneSpecialArc} either the disk $D$ is essential in $F \times I$ as required, or
$\alpha$ is a $\tau_2$--arc.
In this case, $\alpha$ would cut off a disk from $F'$. This disk might contain other arcs of $D' \cap F'$.
Boundary compressing $D'$ along this disk would reduce $|D' \cap F'|$, creating at least two disks, at least one of which would contradict the minimality condition in the choice of $D'$.
Therefore $D$ is essential, and so is a special disk.
\end{proof}

\begin{Lemma}
If $D$ is an $i$-special disk in $F \times I$ for some $i \in \set{0,1}$ then $\bd D$ contains at least one $\tau_{1-i}$--arc.
\end{Lemma}
\begin{proof}
Without loss of generality, assume $D$ is $1$-special. Perform as many tightening-moves on $D$ as possible. This does not change that $D$ is $1$-special, and does not alter any $\tau_0$--arcs in $\bd D$.
Having done this, we see that $\bd D \cap (F \times \set{0})$ consists only of $\tau_0$--arcs.
As $D$ is essential in $F \times I$, Lemma \ref{Lemma:DisksCannotMissF} implies that there must be at least one arc of $\bd D \cap (F \times \set{0})$ remaining, which is therefore a $\tau_0$--arc.
\end{proof}

Now, consider the vertical product disks $E_0' = \tau_0 \times I$, and $E_1' = \tau_1 \times I$. 
We would like to find an $i$-special disk $D$, for some $i \in \set{0,1}$, such that $\bd D$ and $\bd E_i'$ do not intersect on $F \times \set{1-i}$. Since $\pi(\bd E_i' \cap (F \times \set{1-i})) = \pi(\tau_i)$, and a $\tau_{1-i}$--arc in $\bd D$ projects under $\pi$ to $\pi(\tau_{1-i})$, this would show that $\pi(\tau_0)$ and $\pi(\tau_1)$ are disjoint. 
Such a statement is however not true in general unless we first allow a free isotopy of $\tau_0$ (see Section \ref{section:dehntwists} for more details).

As the remainder of the proof will take place within $F \times I$, the precise choice of monodromy $h$ is of no further significance (its importance lies in the existence of special disks). Recall that $\pi(\tau_0) = h(\pi(\tau_1))$. We will now therefore assume that the monodromy has been isotoped (including along the boundary) to minimize $|\pi(\tau_0) \cap \pi(\tau_1)|$. 
At a minimum, this isotopy must perturb the endpoints of $\tau_0$ so that they no longer coincide with those of $\tau_1$ under $\pi$.
We emphasize that changing $h$ rather than isotoping $\tau_0$ is only for notational convenience.
        
\begin{Definition}
\label{Definition:SizeSpecialDisk}
The \emph{size} of a special disk $D$ is the triple 
$(|\bd D \cap (F \times \set{0,1})|,|D \cap E_j'|, |\bd D \cap \bd E_j' \cap (F \times \set{0,1})|)$, where $D$ is a $j$-special disk. 
We will compare the size of two special disks using the lexicographical order. 
\end{Definition}

That is, we order disks first by the number of $\tau_0$--, $\tau_1$--, extra and special arcs, second by the number of arcs and simple closed curves of intersection with the product disk $E_j'$, and finally by the number of endpoints of these intersection arcs that lie on $F$.
We remark that part of the significance of this particular choice of size is that it allows us to compare 0-special and 1-special disks.

\medskip

It is worth noting that this situation looks similar to that found in Lemma 2.3 of \cite{GodHSSMMS}. It appears that one could conclude immediately that a special disk was \emph{boundary compressible towards} $F \times \set{1}$, and repeat such compressions until one arrived at a product disk. This is the idea of our proof, but we need to show some additional care as we want the arcs of $\bd D \cap (F \times \set{0, 1})$ to stay parallel to $\tau_0$ and $\tau_1$ so that we can conclude something about the tunnel. 

Since we know that special disks exist we may take a special disk with minimal size, and call it $D$. 
If $\bd D$ contains no special arc, then $D$ is $1$-special. Pick a $\tau_1$--arc and call it $\alpha$. On the other hand, if it does contain a special arc then we may assume without loss of generality that $D$ is $1$-special. In this case call the special arc $\alpha$.

\begin{Lemma}\label{Lemma:NoExtraArcs}
There are no extra arcs in $\bd D$.
\end{Lemma}
\begin{proof}
If there is an extra arc in $\bd D$, we can perform a tightening-move. This will reduce the number of extra arcs without changing the number of $\tau_0$--, $\tau_1$-- or special arcs. This therefore reduces the size of $D$, a contradiction.
\end{proof}

Lemma \ref{Lemma:NoExtraArcs} implies that $\bd D \cap (F \times \set{0})$ consists only of $\tau_0$--arcs. Let $E' = E_1'$.
Although it is not necessary, for notational convenience we will continue to assume that, for $i \in \set{0,1}$, all $\tau_i$--arcs are contained within the rectangle $A_i$ and run straight from one component of $A_i \cap \bd F$ to the other. 

\begin{Lemma}\label{Lemma:Tau1ArcsDisjoint}
Every arc of $\bd D$ on $F \times \set{1}$ is disjoint from $\bd E'$.
\end{Lemma}
\begin{proof}
Choose $\varepsilon > 0$ such that $(\bd F \times [1 - \varepsilon, 1)) \cap (\bd D \cup \bd E')$ consists of disjoint embedded arcs that are essential in the half-open annulus $\bd F \times [1 - \varepsilon, 1)$.
Let $F^+ = (F \times \set{1}) \cup (\bd F \times [1 - \varepsilon, 1))$.
Since $\bd D \cap (F \times \set{1})$ contains only $\tau_1$-- and special arcs, there is an isotopy of $\bd D \cap F^+$, fixed on $\bd F^+$, that makes $\bd D$ disjoint from $\bd E'$ on $F \times \set{1}$. 
See Figure \ref{Figure:Tau1ArcsDisjoint}.
Because $\bd E' \cap F^+$ is a single arc, this isotopy can be chosen so that it does not increase $|\bd D \cap \bd E' \cap F^+|$ at any point. 
Note that such an isotopy does not change the type of any arc of $\bd D \cap (F \times \set{1})$.
Therefore this means that the isotopy can be extended to an isotopy of $D$ that does not increase $|D \cap E'|$.
If $\bd D \cap \bd E' \cap (F \times \set{1}) \neq \emptyset$ before the isotopy then the isotopy strictly reduces the size of $D$, which is a contradiction. Thus no such isotopy is required and $\bd D \cap \bd E' \cap (F \times \set{1}) = \emptyset$.
\begin{figure}[hbtp]
\begin{center}
\includegraphics[width=10.5cm]{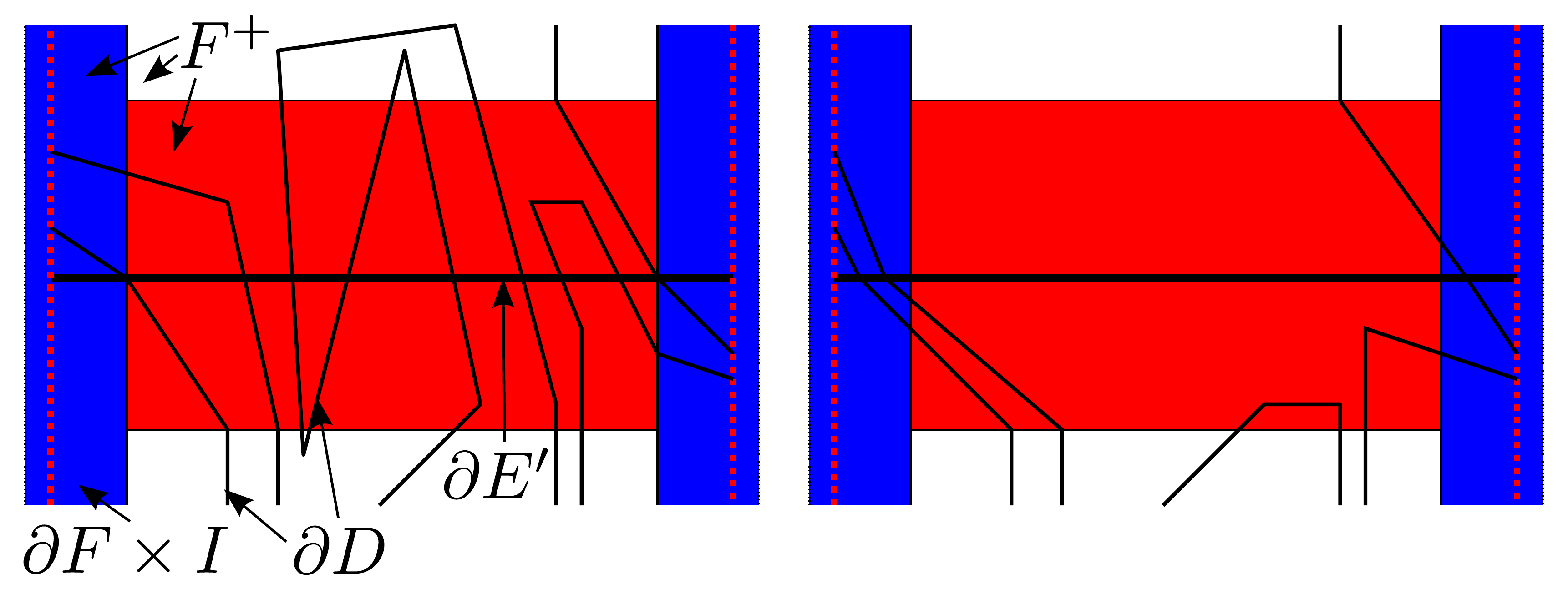}
\caption{Arcs of $\bd D \cap (F \times \set{1})$ can be made disjoint from the arc of $\bd E' \cap (F \times \set{1})$ without increasing $|D \cap E'|$.}
\label{Figure:Tau1ArcsDisjoint}
\end{center}
\end{figure}
\end{proof}

\begin{Lemma}\label{Lemma:IntersectingA0}
We may assume that the endpoints of the arc $\bd E' \cap (F \times \set{0})$ are disjoint from $A_0$, and that every arc of $\bd E' \cap A_0$ connects opposite sides of $A_0$ and intersects each $\tau_0$--arc of $\bd D$ exactly once.
\end{Lemma}
\begin{proof}
By Lemma \ref{Lemma:NoExtraArcs}, $\bd D \cap (F \times \set{0})$ consists only of $\tau_0$--arcs. We have assumed that each of these lies in $A_0$, connecting the two components of $A_0 \cap \bd F$.
By isotoping $A_0 \cap F'$ in $(F \times \set{0}) \rmv \bd D$, we may assume that $\bd E'$ is transverse to $\bd A_0$.

Consider the arcs of $\bd E' \cap A_0$.
Each of the two sides of $A_0$ on $\bd F$ contains at most one endpoint of these arcs. All other endpoints must lie on the two components of $A_0 \cap F'$.
Choose $\varepsilon > 0$ such that $((A_0 \cap \bd F) \times (0, \varepsilon]) \cap (\bd D \cup \bd E')$ consists of disjoint embedded arcs each having one endpoint on $(A_0 \cap \bd F) \times \set{0}$ and one endpoint on $(A_0 \cap \bd F) \times \set{\varepsilon}$.
Let $A_0^+ = A_0 \cup ((A_0 \cap \bd F) \times (0,\varepsilon])$.
As in the proof of Lemma \ref{Lemma:Tau1ArcsDisjoint}, there is an isotopy of $\bd D$ within $A_0^+$, fixed on $\bd A_0^+$, to minimize $|\bd D \cap \bd E' \cap A_0|$, and moreover this isotopy can be chosen so that it extends to an isotopy of $D$ that does not increase the size of $D$ (see Figure \ref{Figure:IntersectionOnA0}).
Again, if this isotopy strictly reduced $|\bd D \cap \bd E' \cap A_0|$ then it would strictly reduce the size of $D$, contradicting that $D$ was chosen to have minimal size. 
Therefore no such isotopy is needed, and the arcs of $\bd E' \cap A_0^+$ have minimal intersection in $A_0$ with the arcs of $\bd D \cap A_0^+$.

Let $\gamma$ be an arc of $\bd E' \cap A_0^+$.
If the endpoints of $\gamma$ lie on distinct components of $A_0 \cap F'$ then we see that $\gamma$ intersects each arc of $\bd D \cap A_0$ exactly once, and because $|\bd D \cap (F \times \set{0})|$ has not increased we know that this intersection occurs within $A_0$.
If the endpoints of $\gamma$ lie on the same component of $A_0 \cap F'$ then we find that $\gamma$ is disjoint from $\bd D$. In this case we may isotope $A_0 \cap F'$ to remove $\gamma$ from $\bd E' \cap A_0$ without affecting $\bd D \cap A_0$ (again see Figure \ref{Figure:IntersectionOnA0}).
If $\gamma$ has one endpoint on $A_0 \cap \bd F$ and the other on $A_0 \cap F'$ then $\gamma \cap A_0$ is disjoint from $\bd D \cap A_0$, and again we may isotope $\bd A_0$ to remove $\gamma$ from $\bd E' \cap A_0$.
Finally suppose that $\gamma$ has both endpoints on components of $A_0 \cap \bd F$. Then $\gamma$ is a $\tau_1$--arc. 
Since $\pi(\gamma)=\pi(\tau_0)$ this shows that $\tau$ and $h(\tau)$ are isotopic in $F$, and in this case the proof of Proposition \ref{Proposition:TunnelCleanAfterIsotopy} is complete.
\begin{figure}[hbtp]
\begin{center}
\includegraphics[width=10.5cm]{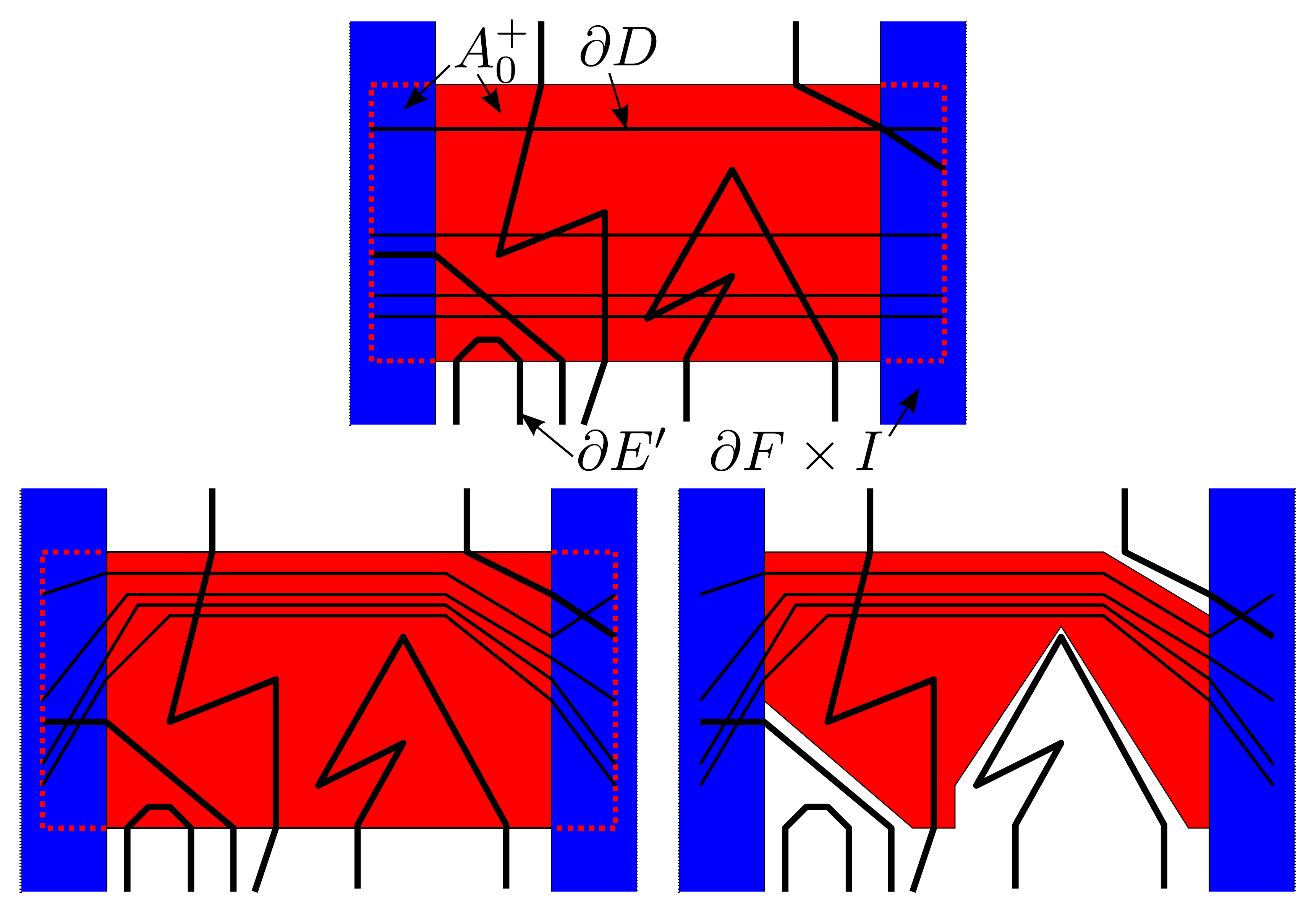}
\caption{$|\bd D \cap \bd E' \cap A_0|$ and $|A_0 \cap F' \cap \bd E'|$ can be minimized without increasing $|D \cap E'|$.}
\label{Figure:IntersectionOnA0}
\end{center}
\end{figure}
\end{proof}

\begin{Lemma}\label{Lemma:NoDisksOrTriangles}
Let $\gamma$ be an arc of $\bd E' \cap (F' \times \set{0})$.
If $\gamma$ has both endpoints on $A_0 \cap F'$ then $\gamma$ does not co-bound a disk in $F'$ with $A_0 \cap F'$.
If $\gamma$ has one endpoint on $A_0 \cap F'$ and one on $\bd F \rmv A_0$ then $\gamma$ does not cut off from $F'$ a disk whose boundary consists of $\gamma$, a single sub-arc of $A_0 \cap F'$ and a single sub-arc of $\bd F \rmv A_0$.
\end{Lemma}
\begin{proof}
Given Lemma \ref{Lemma:IntersectingA0}, this follows immediately from the minimality of $|\pi(\tau_0) \cap \pi(\tau_1)|$
(see Figure \ref{Figure:NoDisksOrTriangles}).
\begin{figure}[hbtp]
\begin{center}
\includegraphics[width=5.25cm]{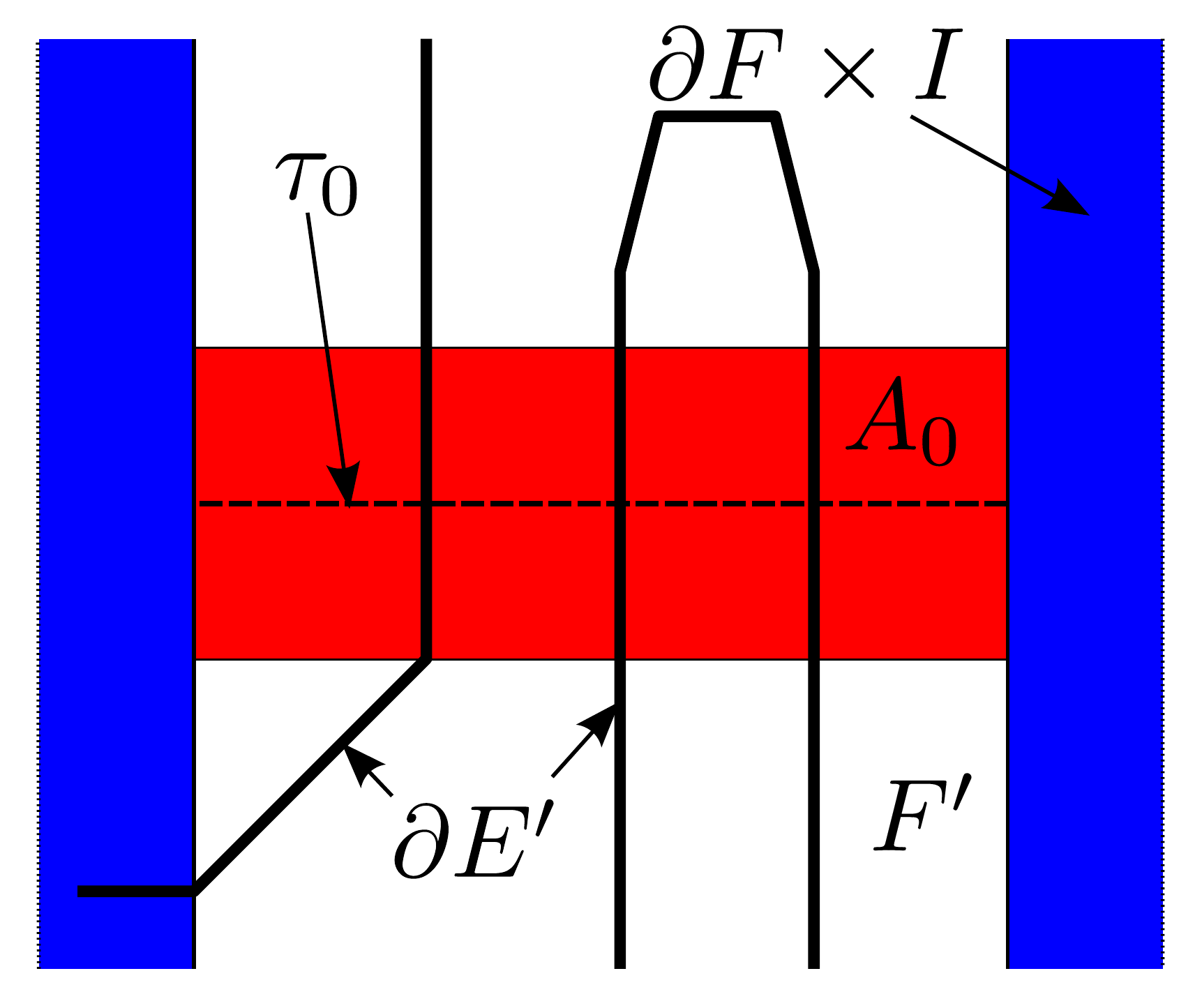}
\caption{Arcs of $\bd E' \cap F'$ do not cut off certain types of disk.}
\label{Figure:NoDisksOrTriangles}
\end{center}
\end{figure}
\end{proof}

Now consider $D \cap E'$. By innermost disk arguments, any simple closed curves of intersection could be removed, since $F \times I$ is irreducible. Thus, since $D$ has minimal size, the intersection consists of arcs. 
From Lemma \ref{Lemma:Tau1ArcsDisjoint} we know that none of these intersection arcs have endpoints on $F \times \set{1}$.
We will show that there are also no arcs of intersection with an endpoint on $F \times \set{0}$. There are three types of arcs that we will be concerned with: type 0 will be arcs with both endpoints on the same component of $\bd E' \cap (\bd F \times I)$; type I will be those with one endpoint on $F \times \set{0}$, and the other on $\bd F \times I$; type II will be arcs with both endpoints incident to $F \times \set{0}$ (see Figure \ref{Figure:VerticalDisk}).
Showing that none of these arcs exist, and hence $\bd D \cap \bd E' \cap (F \times \set{0}) = \emptyset$, will complete the proof of Proposition \ref{Proposition:TunnelCleanAfterIsotopy}.
\begin{figure}[hbtp]
\begin{center}
\includegraphics[width=7cm]{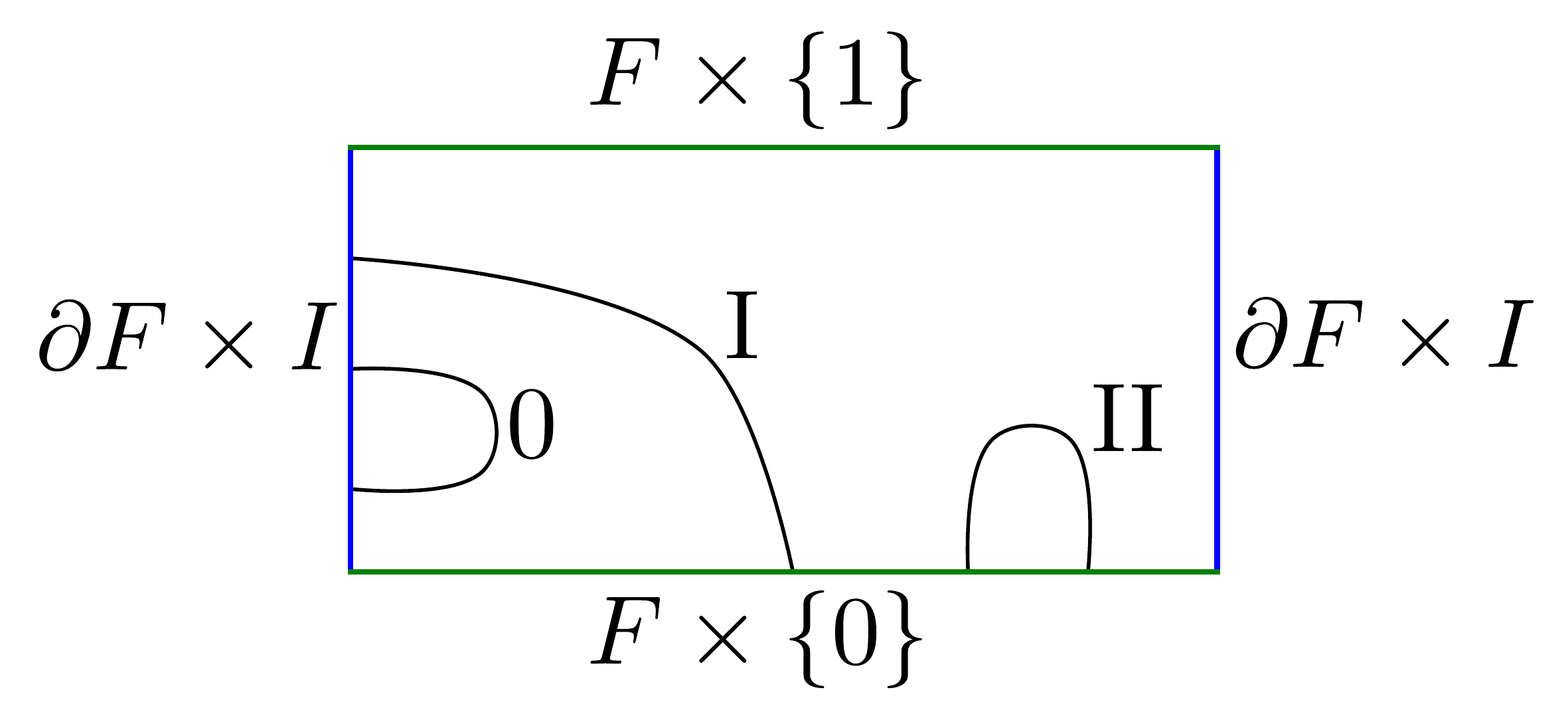}
\caption{Arcs of $D \cap E'$ of type 0, type I and type II in $E'$.}
\label{Figure:VerticalDisk}
\end{center}
\end{figure}

\subsection{Arcs of type 0}

Suppose there is an arc of $D \cap E'$ with both endpoints on the same component of $E' \cap (\bd F \times I)$.
Choose such an arc that is outermost in $E'$, and let $E$ be the subdisk of $E'$ it cuts off.
Compress $D$ along $E$, reducing $|D \cap E'|$ without altering the arcs of $\bd D \cap (F \times \set{0,1})$. This gives two disks, $D^*$ and $D^{**}$. 
Take $D^*$ to be the one containing $\alpha$ in its boundary.
At least one of $D^*$ and $D^{**}$ is essential, and neither has more than one special arc or any bad arcs in its boundary. 
In addition, $|D^* \cap (F \times \set{0,1})| \leq |D \cap (F \times \set{0,1})|$ and $|D^* \cap E'| < |D \cap E'|$, while $|D^{**} \cap (F \times \set{0,1})| < |D \cap (F \times \set{0,1})|$.
Therefore at least one of $D^*$ and $D^{**}$ is special and has smaller size than $D$, which is a contradiction.
Hence no arcs of type 0 exist.

\subsection{Arcs of type II}
If there is an arc of type II, then there is an arc of type II that is outermost in $E'$. Call this arc $\delta$, and call the subdisk of $E'$ that it cuts off $E$. Let $\gamma = \bd E \rmv \delta$. 
Boundary compressing $D$ along $E$ reduces $|D \cap E'|$ and gives two disks, $D^*$ and $D^{**}$, at least one of which is essential. Take $D^*$ to be the resulting disk containing $\alpha$ in its boundary. The endpoints of $\gamma$ must both be on $\tau_0$--arcs. 

First suppose that $\gamma \subset A_0$. Then by Lemma \ref{Lemma:IntersectingA0} we know that the endpoints of $\gamma$ lie on distinct $\tau_0$--arcs of $\bd D$. 
Let $\beta^*$ and $\beta^{**}$ be the sub-arcs of $\bd D^* \cap (F \times \set{0})$ and $\bd D^{**} \cap (F \times \set{0})$ respectively that contain copies of $\gamma$.
Then $\beta^*$ and $\beta^{**}$ are both extra arcs (see Figure \ref{Figure:gammaIa}), so neither $D^*$ nor $D^{**}$ has any bad arcs or more than one special arc in its boundary.
Moreover, it is again the case that $|D^* \cap (F \times \set{0,1})| \leq |D \cap (F \times \set{0,1})|$ and $|D^* \cap E'| < |D \cap E'|$, while $|D^{**} \cap (F \times \set{0,1})| < |D \cap (F \times \set{0,1})|$.
This tells us that at least one of $D^*$ and $D^{**}$ is special and has smaller size than $D$, a contradiction.
\begin{figure}[htbp]
\begin{center}
\includegraphics[width=10.5cm]{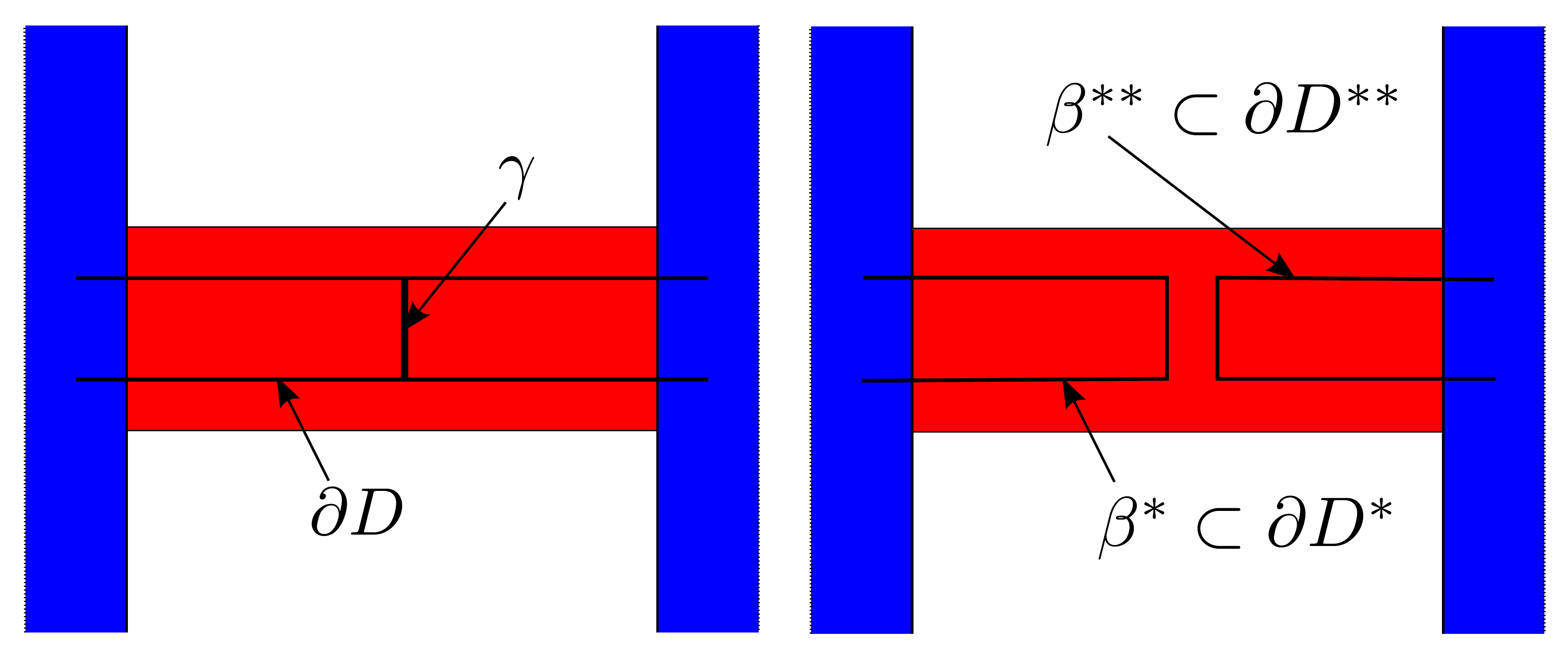}
\caption{If $\gamma \subset A_0$ then $\beta^*$ and $\beta^{**}$ are extra arcs.}
\label{Figure:gammaIa}
\end{center}
\end{figure} 

Now assume instead that $\gamma \not \subset A_0$.
Then it runs between two $\tau_0$--arcs that are outermost in $A_0$.
That is, $\gamma$ runs from a sub-arc of $\bd D$, across one of the sides of $\bd A_0$ incident to $F'$, through $F'$, then across a side of $\bd A_0$ and to another sub-arc of $\bd D$. There are, then, two things which might happen. Either $\gamma$ returns to the same side of $\bd A_0$ (see Figure \ref{Figure:gammaIIsame}), or it returns to the other side of $\bd A_0$ (see Figure \ref{Figure:gammaIIdifferent}).

If $\gamma$ returns to the same side of $\bd A_0$, then both endpoints must be incident to the same component of $\bd D \cap A_0$ (see Figure \ref{Figure:gammaIIsame}) and $\bd D^{**}$ is a simple closed curve in $F \times \set{0}$. 
Lemma \ref{Lemma:NoDisksOrTriangles} shows that $\bd D^{**}$ does not bound a disk in $F$,
so this means that $D^{**}$ is a compression disk for $F$, contradicting that $F \times \set{0}$ is incompressible in $F \times I$. 
\begin{figure}[htbp]
\begin{center}
\includegraphics[width=10.5cm]{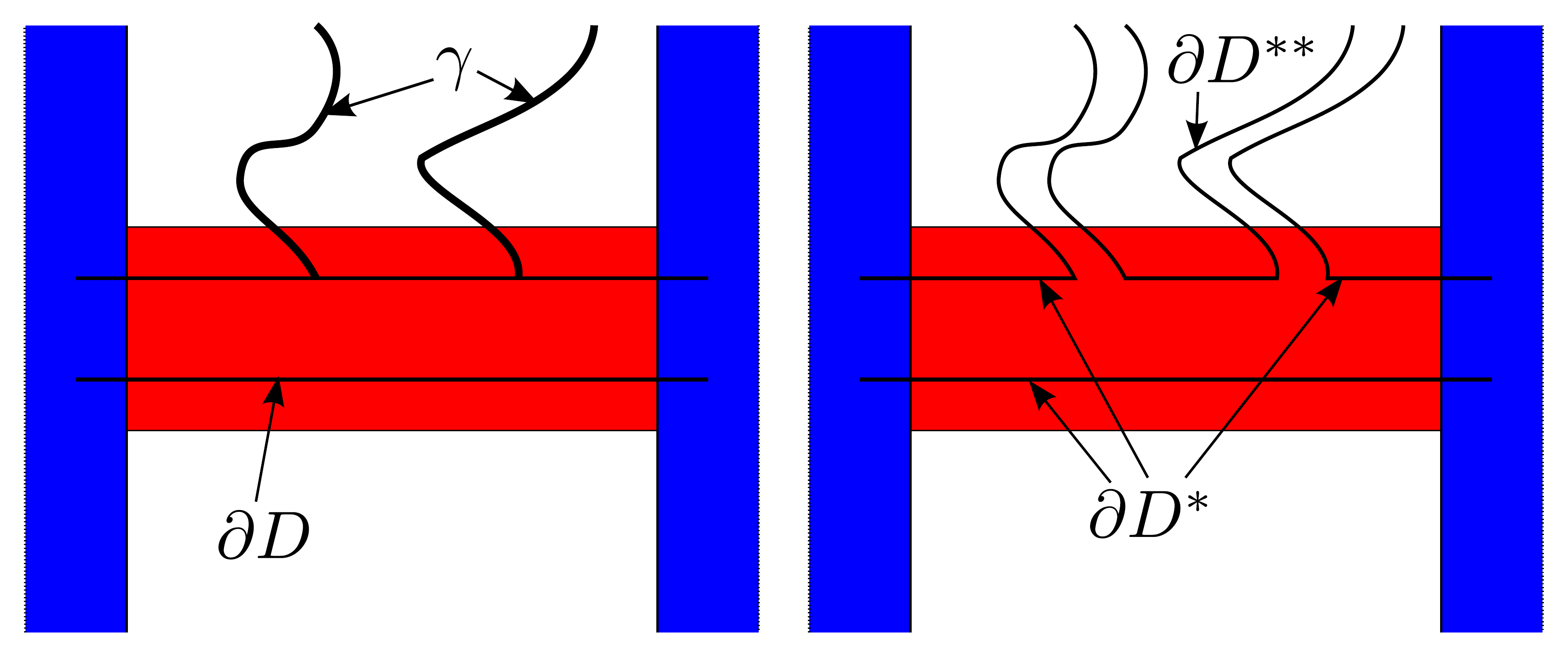}
\caption{If $\gamma \not \subset A_0$, and returns to $A_0$ on the same side, then $D^{**}$ is a compression disk for $F$.}
\label{Figure:gammaIIsame}
\end{center}
\end{figure}  

If $\gamma$ returns to the other side of $\bd A_0$, then the orientation on $D$ implies that there are at least two $\tau_0$--arcs in $\bd D$.
Let $\beta^*$ and $\beta^{**}$ be the sub-arcs of $\bd D^*$ and $\bd D^{**}$ respectively that contain copies of $\gamma$
(see Figure \ref{Figure:gammaIIdifferent}).
\begin{figure}[htbp]
\begin{center}
\includegraphics[width=10.5cm]{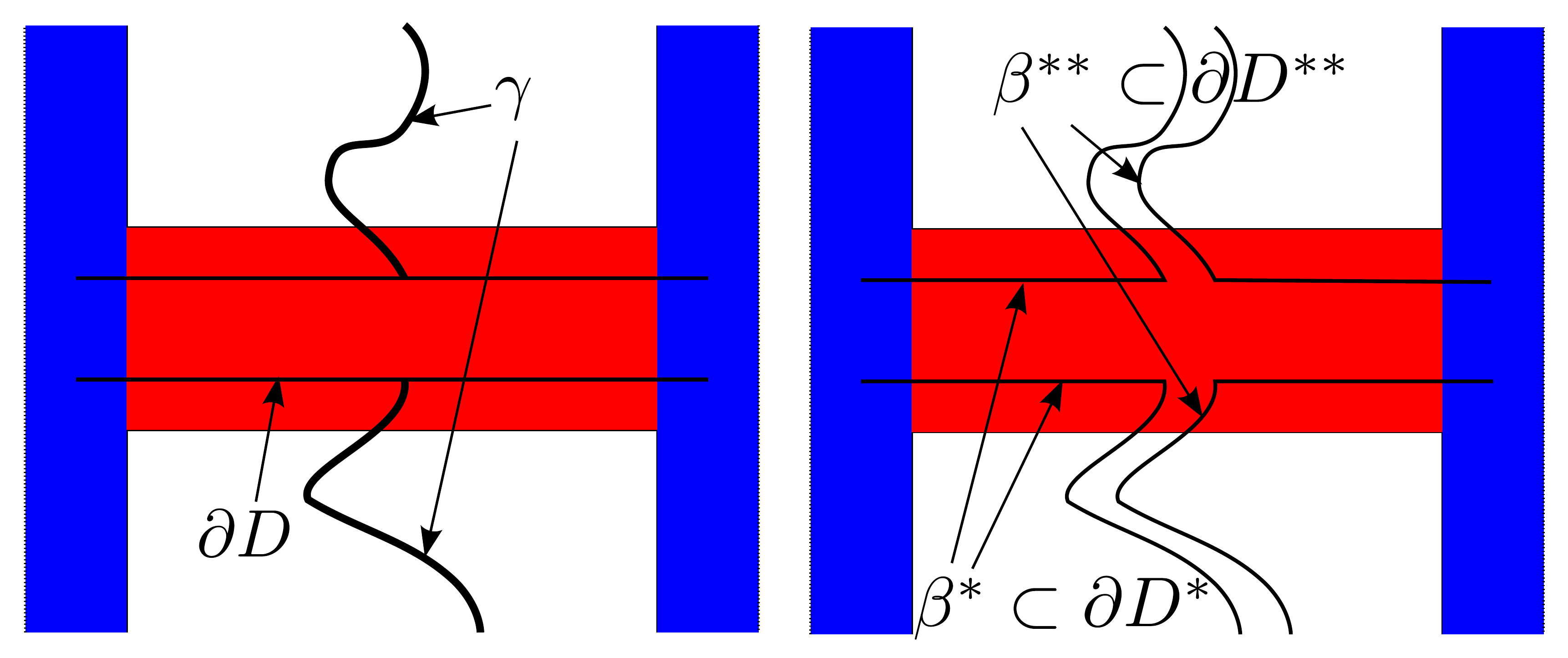}
\caption{If $\gamma \not \subset A_0$ then $|D^* \cap (F \times \set{0,1})| + |D^{**} \cap (F \times \set{0,1})| = |D \cap (F \times \set{0,1})|$.}
\label{Figure:gammaIIdifferent}
\end{center}
\end{figure}  
There are no bad arcs in either $\bd D^*$ or $\bd D^{**}$, and there is at most one special arc in $\bd D^{**}$.
As before, $|D^* \cap (F \times \set{0,1})| \leq |D \cap (F \times \set{0,1})|$ and $|D^* \cap E'| < |D \cap E'|$, while $|D^{**} \cap (F \times \set{0,1})| < |D \cap (F \times \set{0,1})|$.

If $D^{**}$ is essential then it is a special disk with smaller size than $D$, which is a contradiction. Suppose otherwise. Then $D^*$ is essential. Additionally, by Lemma \ref{Lemma:OneSpecialArc}, $\beta^{**}$ is either an extra arc, a $\tau_0$--arc or a $\tau_2$--arc.

If $\beta^{**}$ is a $\tau_0$--arc then $|\bd F|=1$, since the endpoints of $\beta^{**}$ lie on the same component of $\bd F$.
However, there is an arc parallel to $\tau_0$ in $A_0$ that is disjoint from $\beta^{**}$ and whose endpoints interleave on $\bd F$ with those of $\beta^{**}$.
It is therefore impossible that these two arcs together bound a disk in $F$. This shows that $\beta^{**}$ is not a $\tau_0$--arc.

If $\beta^{**}$ is a $\tau_2$--arc then $|\bd F|=2$ and $\beta^*$ is an extra arc. Thus $D^*$ is a special disk with smaller size than $D$, a contradiction.

If $\beta^{**}$ is an extra arc then $|\bd F|=2$ and $\beta^*$ is a $\tau_2$--arc.
Let $F^*$ be the subdisk of $F'$ that $\beta^*$ cuts off.
Now, $(\bd E' \cap (F \times \set{0})) \rmv \gamma$ consists of two arcs; call these $\gamma'$ and $\gamma''$. From their endpoints that meet $\gamma$, both $\gamma'$ and $\gamma''$ run to the opposite side of $A_0$, by Lemma \ref{Lemma:IntersectingA0}. At this point, therefore, one of $\gamma'$ and $\gamma''$ lies closer than the other in $A_0$ to the component of $A_0 \cap \bd F$ containing the endpoints of $\beta^*$.
Take this to be $\gamma'$. See Figure \ref{Figure:LinkTwistingII}.

Consider the path of $\gamma''$ from the endpoint that meets $\gamma$. When it first leaves $A_0$, $\gamma''$ enters the disk $F^*$. As we continue to follow its path, it can either end on the component of $\bd F \times \set{0}$ that contains the endpoints of $\beta^{**}$ or else return to $A_0 \cap F'$ (necessarily on the other side, by Lemma \ref{Lemma:NoDisksOrTriangles}). We see, therefore, that $\gamma''$ spirals around one boundary component of $F \times \set{0}$ some number of times before ending on this component of $\bd F$.
Consider the final section of $\gamma''$, from where it last leaves $A_0$ to where it reaches $\bd F$.
This cuts off a disk from $F'$, the remainder of whose boundary consists of a single sub-arc of $A_0 \cap F'$ and a single sub-arc of $\bd F \rmv A_0$. This contradicts Lemma \ref{Lemma:NoDisksOrTriangles}. It is therefore not possible that $\beta^{**}$ is an extra arc.
\begin{figure}[htbp]
\begin{center}
\includegraphics[width=9cm]{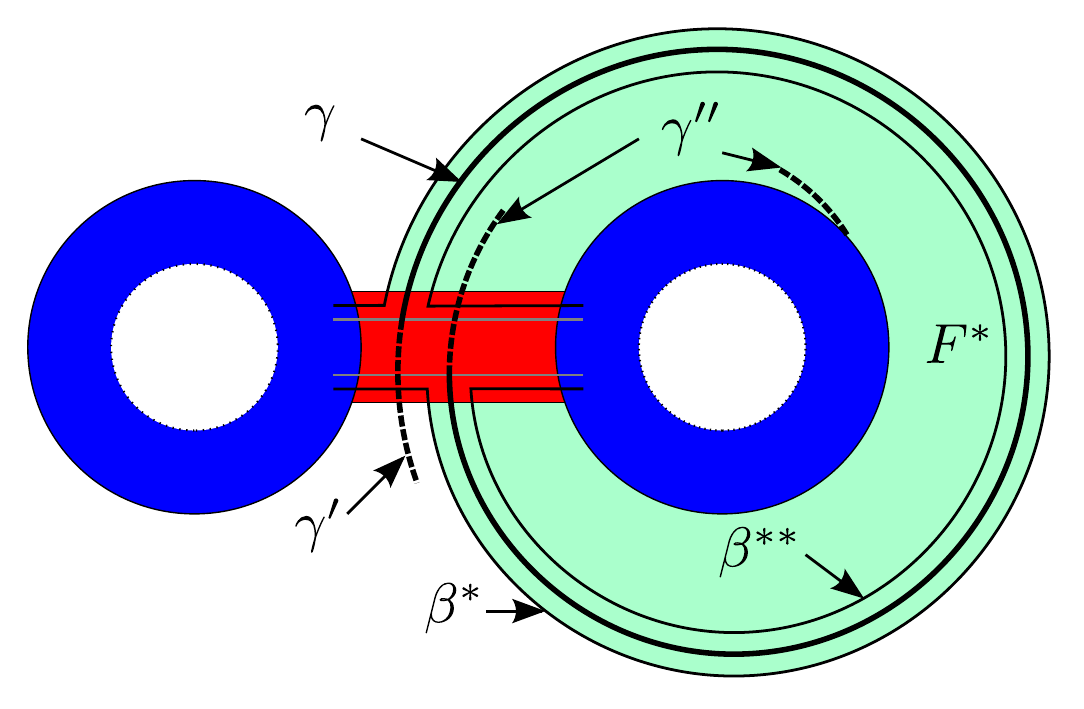}
\caption{If $\beta^{**}$ is an extra arc then $|\bd F|=2$ and $\gamma''$ spirals around one component of $\bd F$.}
\label{Figure:LinkTwistingII}
\end{center}
\end{figure} 

Thus, we conclude that there are no arcs of type II in $D \cap E'$.

\subsection{Arcs of type I}

Since we now know there are no arcs of types 0 or II, if there are arcs of type I then one of them is outermost in $E'$. Again, call one of these arcs $\delta$, and call the subdisk of $E'$ that it cuts off $E$. Let $\gamma = \bd E \rmv \delta$.
Then $\gamma$ consists of two sub-arcs. Let $\gamma_0 = \gamma \cap (F \times \set{0})$, and $\gamma_\bd = \gamma \cap (\bd F \times I)$. Observe that $\gamma_0$ has one endpoint on a $\tau_0$--arc of $\bd D$ and the other end on $\bd F \rmv A_0$, given Lemma \ref{Lemma:IntersectingA0}.
Note that, since $\gamma_0$ is disjoint on its interior from $\bd D$, Lemma \ref{Lemma:IntersectingA0} also tells us that $\gamma_0 \cap A_0$ is a single sub-arc of $\gamma_0$. 

As before, boundary compressing $D$ along $E$ results in two disks, $D^*$ and $D^{**}$, at least one of which is essential. Again let $D^*$ be the one that contains $\alpha$ in its boundary. 
Let $\beta^*$ and $\beta^{**}$ be the sub-arcs of $\bd D^*$ and $\bd D^{**}$ respectively that contain copies of $\gamma_0$.
As previously, $|D^* \cap E'| < |D \cap E'|$, neither $\bd D^*$ nor $\bd D^{**}$ contains any bad arcs, and $\bd D^{**}$ contains at most one special arc.
Now $|\bd D^* \cap (F \times \set{0,1})| + |\bd D^{**} \cap (F \times \set{0,1})| = |\bd D \cap (F \times \set{0,1})| + 1$. In addition, $|\bd D^* \cap (F \times \set{0,1})| \geq 2$ while $ |\bd D^{**} \cap (F \times \set{0,1})|\geq 1$.
Therefore $|\bd D^* \cap (F \times \set{0,1})| \leq |\bd D \cap (F \times \set{0,1})|$ and $|\bd D^{**} \cap (F \times \set{0,1})| < |\bd D \cap (F \times \set{0,1})|$.

From Lemma \ref{Lemma:NoDisksOrTriangles}, we know that neither $\beta^*$ nor $\beta^{**}$ is an extra arc.
If $D^{**}$ is essential then it is a special disk with smaller size than $D$, which is a contradiction. Suppose otherwise. Then $D^*$ is essential. Additionally, by Lemma \ref{Lemma:OneSpecialArc}, $\beta^{**}$ is either a $\tau_0$--arc or a $\tau_2$--arc.

If $\beta^{**}$ is a $\tau_2$--arc then $|\bd F|=2$ and $\beta^*$ is a $\tau_0$--arc. Thus $D^*$ is a special disk that is smaller than $D$, a contradiction.

If $\beta^{**}$ is a $\tau_0$--arc then,
as $\beta^{**}$ is disjoint from a copy of $\tau_0$ in $A_0$, together these arcs bound a disk in $F$.
If $|\bd F|=1$, the presence of this disk tells us that the endpoints of $\beta^{**}$ do not interleave on $\bd F$ with those of $\tau_0$. Therefore the disk contains $\beta^*$ and $\beta^*$ is an extra arc, a contradiction.

It remains only to consider the case that $|\bd F|=2$, when $\beta^{**}$ has its endpoints on distinct components of $\bd F$.  Again, if the disk between $\beta^{**}$ and $\tau_0$ contains $\beta^*$ then $\beta^*$ is an extra arc, a contradiction.
Accordingly, the disk does not contain $\beta^*$, and $\beta^*$ is a $\tau_2$--arc, cutting off from $F'$ a disk $F^*$. 
Let $\gamma_0' = (\bd E' \cap (F \times \set{0})) \rmv \gamma_0$.
This is an arc with one endpoint on a $\tau_0$--arc of $\bd D$, where it meets $\gamma_0$, and the other endpoint on $\bd F \rmv A_0$. Given the definition of $E'$, this endpoint lies on the opposite component of $\bd F$ to the other endpoint of $\gamma_0$. That is, $\gamma_0'$ does not meet the same component of $\bd F$ as $\beta^*$ does. See Figure \ref{Figure:LinkTwistingI}.
Consider the path of $\gamma_0'$ from where it meets $\gamma_0$.
It first runs through $A_0$, and passes through $A_0 \cap F'$ into the disk $F^*$. 
As we continue to follow its path, it can either end on $\bd F$ or else return to $A_0 \cap F'$ (necessarily on the other side, by Lemma \ref{Lemma:NoDisksOrTriangles}). We see, that, like the arc $\gamma''$ above, $\gamma_0'$ spirals around one boundary component of $F \times \set{0}$ some number of times before ending on the same component of $\bd F$.
Consider the final section of $\gamma_0'$, from where it last leaves $A_0$ to where it reaches $\bd F$.
This cuts off a disk from $F'$, the remainder of whose boundary consists of a single sub-arc of $A_0 \cap F'$ and a single sub-arc of $\bd F \rmv A_0$. This contradicts Lemma \ref{Lemma:NoDisksOrTriangles}. 
\begin{figure}[htbp]
\begin{center}
\includegraphics[width=9cm]{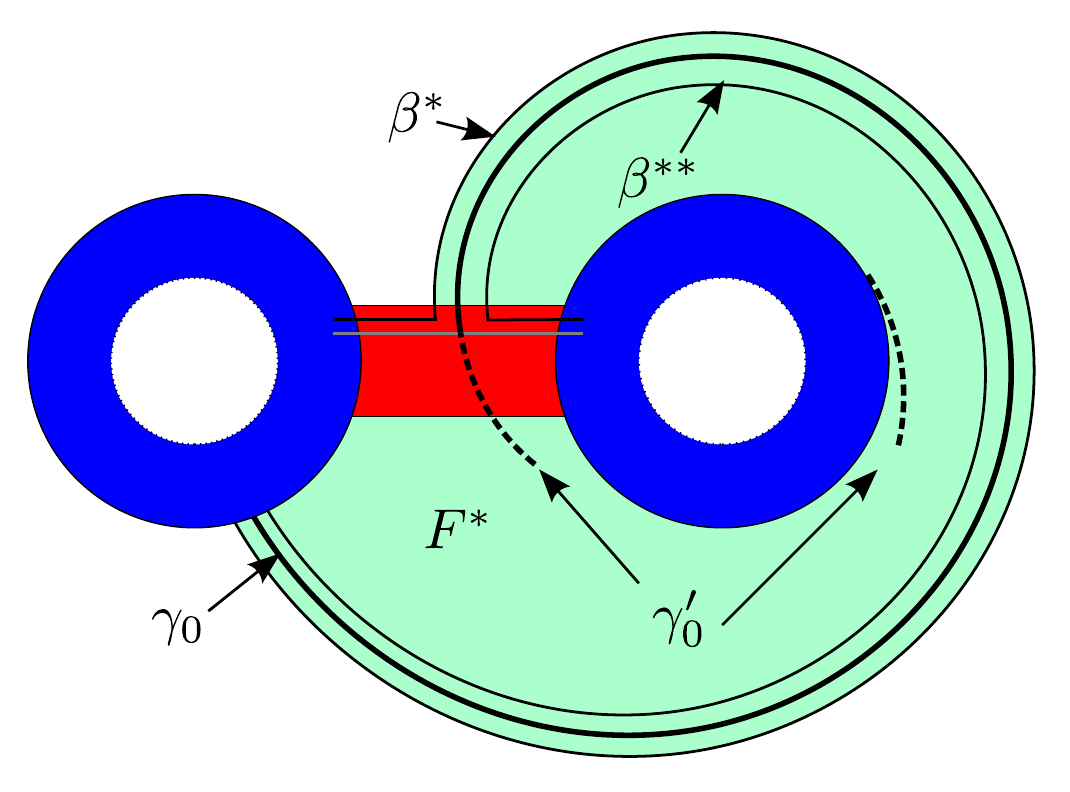}
\caption{If $\beta^{**}$ is a $\tau_0$--arc and $|\bd F|=2$ then $\gamma_0'$ spirals around one component of $\bd F$.}
\label{Figure:LinkTwistingI}
\end{center}
\end{figure} 

Thus, there are no arcs of type I. This completes the proof of Proposition \ref{Proposition:TunnelCleanAfterIsotopy}.
\end{proof}     

\TheoremTunnelAlmostClean*

\begin{proof}
By taking a small enough neighborhood $n(K \cup \tau)$ of $K \cup \tau$ in $S^3$, we may assume that $F \cap n(K \cup \tau)$ is a thrice-punctured sphere, and the Heegaard surface $S = \bd n(K \cup \tau)$ intersects $F$ in exactly two non-separating curves if $K$ is a knot, and exactly one non-separating curve if $K$ is a two-component link, 
as shown in Figure \ref{Figure:HeegaardSurfaceandFiber}.

\begin{figure}[hbtp]
\begin{center}
\includegraphics[width=11cm]{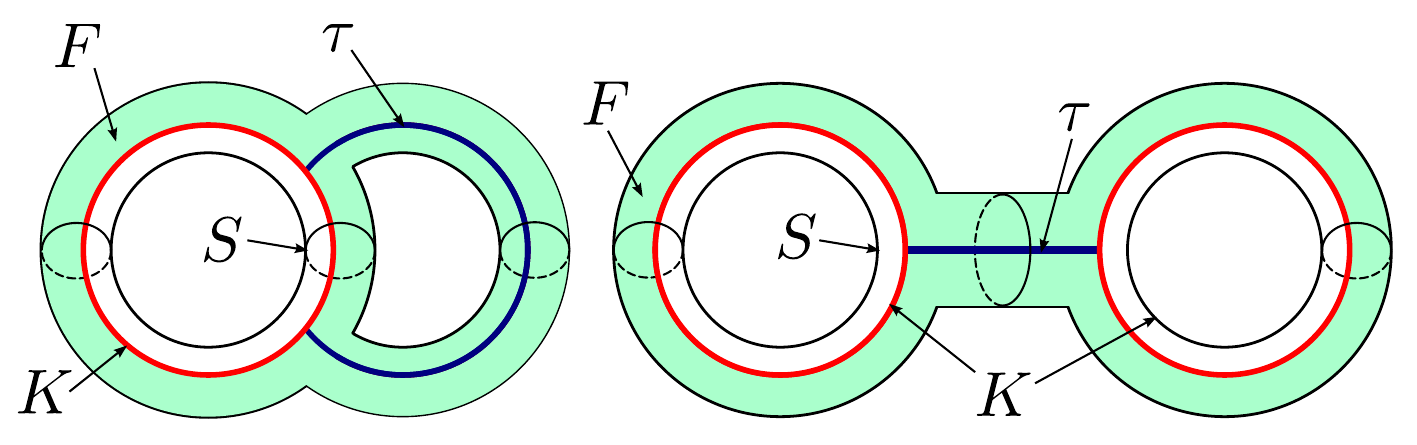}
\caption{For a small enough neighborhood, $F$ intersects $S$ in one or two curves.}
\label{Figure:HeegaardSurfaceandFiber}
\end{center}
\end{figure}       

Recall that $(S^3\rmv n(K))|F \cong F \times I$. 
Since $K$ is not the unknot, $F$ is not a disk. If $F$ is an annulus then $\tau$ is in the unique free isotopy class of essential arcs in $F$, and Theorem \ref{Theorem:TunnelAlmostClean} is immediate.
Any tunnel number one link has at most two link components, so $|\bd F|\leq 2$.
From the definition of an unknotting tunnel, the hypotheses imply that $S^3 \rmv n(K \cup \tau)$ is a genus two handlebody whereas $S^3 \rmv n(K)$ is not a handlebody. 
Thus Theorem \ref{Theorem:TunnelAlmostClean} follows from Proposition \ref{Proposition:TunnelCleanAfterIsotopy}.
\end{proof}

    
\section{Boundary Twisting and Fractional Dehn Twists}\label{section:dehntwists}

In this section, we will discuss why the free isotopy mentioned in Theorem \ref{Theorem:TunnelAlmostClean} is necessary, and why a stronger claim about tunnels being clean cannot be made in general.

First, consider a surface bundle $M = (F \times I)/h$ as in Proposition \ref{Proposition:TunnelCleanAfterIsotopy}, and suppose $M$ is tunnel number one (i.e.\ that there is an arc $\tau \subset F$ such that $M \rmv n(\tau)$ is a genus two handlebody). Let $T_\bd$ be a Dehn twist along a curve in $F$ that is parallel to a component of $\bd F$. Then for all $n \in \mathbb{Z}$, the maps $h$ and $T_\bd^n \circ h$ are freely isotopic, so that $(F \times I)/h \cong (F \times I)/(T_\bd^n \circ h)$. In fact, $\tau \subset F$ is still a tunnel for $(F \times I)/(T_\bd^n \circ h)$. However, even if $\tau$ is clean with respect to $h$, there will be intersections between $\tau$ and $(T_\bd^n \circ h)(\tau)$ in a neighborhood of $\bd F$ for all sufficiently high values of $|n|$. These intersections can be removed by freely isotoping $(T_\bd^n \circ h)(\tau)$ independently of $\tau$, but then the arc does not correspond to the image of $\tau$ under the map $(T_\bd^n \circ h)$. Recall that these twists do affect the meridian(s) of the link, and can be viewed as changing the ambient 3-manifold in which the fibered link sits.

One might hope that this type of indeterminacy would improve if we restrict our attention to knots and links in $S^3$, as this would specify the representative monodromy map by determining the meridian(s). We next, therefore, consider an example in $S^3$, suggested to the authors by Ken Baker. Suppose $\tau$ is the upper (or lower) tunnel for a fibered 2-bridge knot $K$ (see \cite{KobCUT2BK}), sitting in a fiber surface $F$ as a clean arc such that $h(\tau) \neq \tau$. Now, perform a Hopf plumbing along an arc that is parallel into $\bd F$, but has endpoints interleaved on $K$ with those of $\tau$. The result is $K \# L$, where $L$ is a Hopf link, and has a monodromy map $h'$ that is a composition of $h$ with a Dehn twist around the core curve of the Hopf band. The choice of sign for the Hopf band determines the orientation on the link, as well as the sign of the Dehn twist. Either way, $\tau$ is a tunnel of $K \# L$, since $\tau$ together with the unknotted component of the link is actually equivalent to one of the dual upper tunnels for $K$ (see \cite{KobCUT2BK}). Although one choice results in a monodromy under which $\tau$ is still clean, the other results in a monodromy under which it is not, since the extra twist forces an intersection between $\tau$ and $h'(\tau)$ in a neighborhood of the boundary of the fiber. See Figure \ref{Figure:plumbedHopfband}.
\begin{figure}[hbtp]
\begin{center}
\includegraphics[height=7cm]{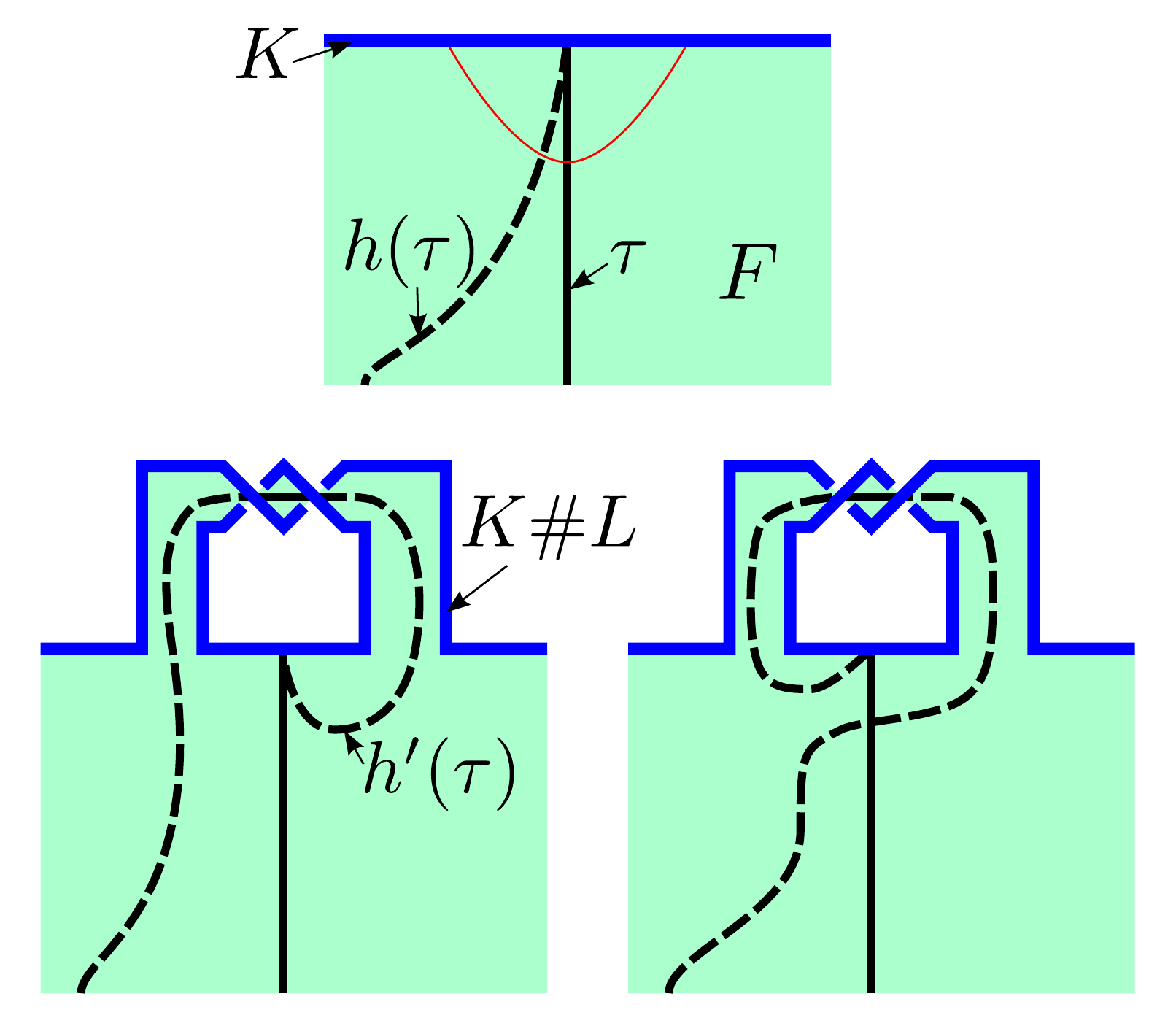}
\caption{One choice of Hopf plumbing gives a clean tunnel while the other does not.}
\label{Figure:plumbedHopfband}
\end{center}
\end{figure}

In fact, it is not only in the case of a connected sum with a Hopf link that this complication with boundary twisting arises. Kai Ishihara pointed out to the authors that if $L$ is a tunnel number one, fibered, two-component link in $S^3$ with one trivial component, $K$, and linking number $\pm 2$, $\pm 1$, or 0, then modifying the monodromy by $n$ ($n = \mp 1, n = \mp 2$, or $n$ arbitrary, respectively) Dehn twists along a curve in the fiber parallel to $K$ corresponds to performing Stallings twists, and produces tunnel number one, fibered links in $S^3$, each with tunnels that intersect their monodromy images (several times) in a neighborhood of the boundary of the fiber. 

One such example is the Whitehead link, which has linking number zero, and is additionally hyperbolic. Figure \ref{Figure:twistedWhitehead} (left) shows the link resulting from twisting $n = 3$ times around one of the components of the Whitehead link, along with a tunnel, $\tau$, for this link. One can check that the surface illustrated is a fiber (since it is genus one, i.e. minimal genus), and that $\tau$ is a tunnel for the link. In fact, $\tau$ is not clean, as the image of $\tau$ under the monodromy is indicated. Alternatively, one can see that $\tau$ cannot be clean because cutting the fiber surface along the tunnel arc produces a surface whose boundary is the $5_2$ knot. If $\tau$ were clean and alternating, then it would correspond to a plumbed Hopf band, the de-plumbing of which would result in a genus one fiber surface with a connected boundary, so the boundary would be a trefoil or figure-eight knot. On the other hand, if it were clean and non-alternating, then cutting along $\tau$ would result in a pre-fiber surface (see \cite{KobFLUO}), which itself would be a (genus one) compressible surface, implying that the boundary was the unknot. 

Twisting the same component of the Whitehead link an arbitrary $n$ times also results in a new tunnel number one, fibered link. In Figure \ref{Figure:twistedWhitehead} (right), the light gray arc still indicates a tunnel, and the black train track with weights determines the arc that is the image of this tunnel under the monodromy for this surface.

\begin{figure}[hbtp]
\begin{center}
\includegraphics[height=6.5cm]{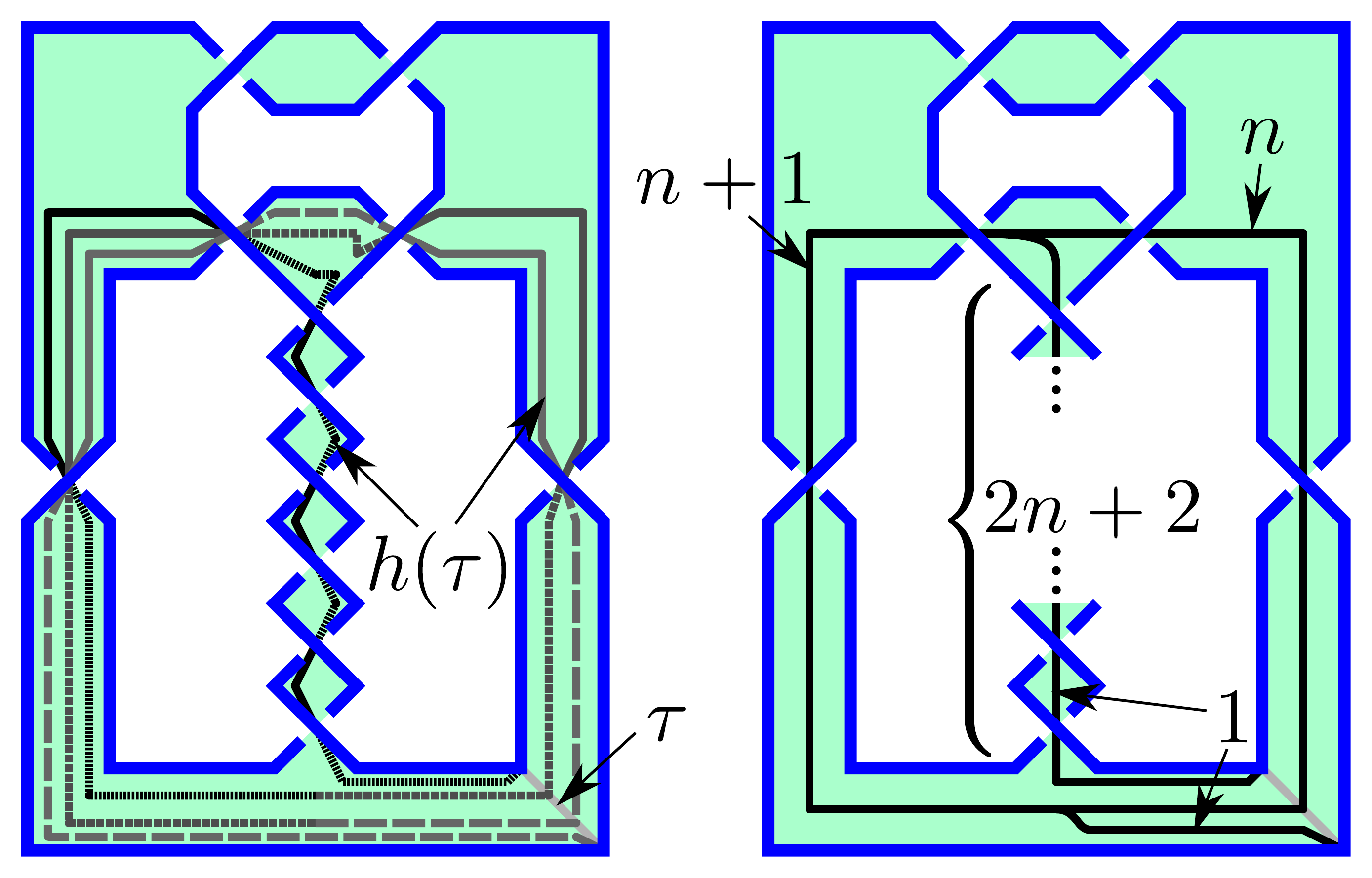}
\caption{A hyperbolic, tunnel number one, fibered link with an unclean tunnel obtained by twisting the Whitehead link around an unknotted component $n=3$ (left) or $n \geq 1$ (right) times.}
\label{Figure:twistedWhitehead}
\end{center}
\end{figure}

In light of the examples discussed above, it is reasonable to hope that if a tunnel number one, fibered link of two components has an unclean tunnel, then one of the components must be unknotted. There are, moreover, no known examples of tunnel number one, fibered \emph{knots} with unclean tunnels. It would also be reasonable to hope that for such knots, the tunnels are always clean. 

Thurston classified automorphisms of a (hyperbolic) surface (see \cite{ThuGDDS} and \cite{CasBleASANT}). Every automorphism $f \colon F \to F$ is freely isotopic to one, $\widetilde{f}$, that is either (1) \emph{reducible}: there is an essential multi-curve $\gamma$, none of whose components are parallel into $\bd F$, so that $\widetilde{f}(\gamma) = \gamma$, set-wise; (2) \emph{periodic}: $\widetilde{f}^n = \Id$ for some $n>0$; or (3) \emph{pseudo-Anosov}: there exist a pair of minimal geodesic laminations, $\Lambda^s$ and $\Lambda^u$, called the stable and unstable laminations respectively, each together with transverse measures, $\mu^s$ and $\mu^u$ respectively, and so that there exists a real number $c > 1$, with $\widetilde{f}(\Lambda^s, \mu^s) = (\Lambda^s, c \cdot \mu^s)$ and $\widetilde{f}(\Lambda^u, \mu^u) = (\Lambda^u, c^{-1} \cdot \mu^u)$. In all cases, we call $\widetilde{f}$ the \emph{Thurston representative} of $f$. (We follow the convention of referring to a map as reducible only if it is not periodic.)

By Thurston's Hyperbolization Theorem a surface bundle over $S^1$ is hyperbolic if and only if the (Thurston representative of the) monodromy map is pseudo-Anosov (see \cite{SulTTGQFVHDF}, \cite{OtaTHPVFD} or \cite{OtaHTFM}). Since the Whitehead link is hyperbolic then, the Thurston representative of its monodromy is correspondingly pseudo-Anosov. Observe that this means the family of examples given in Figure \ref{Figure:twistedWhitehead} are all hyperbolic, since all of their respective monodromies are freely isotopic to the monodromy of the Whitehead link.

We now develop some terminology about \emph{fractional Dehn twists} (for more, see \cite{KaRoFDTKTCT} and \cite{RobTFPSB2}), to explore the question of boundary twisting even further.

Let $p$ and $q$ be relatively prime integers, with $q > 0$. A $\frac{p}{q}$--\emph{fractional Dehn twist} on the annulus $A = \set{r e^{i \theta} \in \mathbb{C} \, | \, 1 \leq r \leq 2}$ is the map $T_{p/q} \colon re^{i\theta} \mapsto re^{i\left(\theta + (r - 1)\frac{2\pi p}{q}\right)}$.

Now, given a surface bundle determined by (bounded) surface $F$ and monodromy $h$, where $h(x) = x$ for all $x \in \bd F$, let $C$ be a boundary component of $F$. If $h$ is reducible, the fixed curves divide $F$ into pieces, and we may focus on the single piece containing $C$. After further cutting if needed, we find a subsurface of $F$ containing $C$ on which $\widetilde{h}$ is either periodic or pseudo-Anosov.

Let $S = F \cup_{C \times \set{1}} (C \times I)$, and extend $h$ by the identity on $C \times I$. Next, isotope $h \cup \Id$ relative to $C$, so that its restriction to $F$ equals $\widetilde{h}$. Call this map $\overline{h}$. If $\widetilde{h}$ is periodic, $\overline{h}$ can be chosen so that $\overline{h} |_{C \times [0, 1]} = T_{p/q} |_{C \times [0, 1]}$ for some $p,q$. In this case, we say that $h$ has \emph{fractional Dehn twist coefficient} $\frac{p}{q}$ with respect to the boundary component $C$.
If instead $\widetilde{h}$ is pseudo-Anosov, then it does not act as an isometry on $C$, so the fractional Dehn twist cannot be defined in exactly the same way. However, the lamination $\Lambda^s$ (and also the lamination $\Lambda^u$) specifies a finite orbit of points on $C$. In this case $\overline{h}$ can be chosen so that $\overline{h} |_{C_{\Lambda} \times [0,1]} = T_{p/q} |_{C_{\Lambda} \times [0,1]}$ for some $p,q$, where $C_{\Lambda}$ is this set of points. Again, $\frac{p}{q}$ is the \emph{fractional Dehn twist coefficient} of $h$. 

When the surface bundle is a knot complement in $S^3$, it is known that the fractional Dehn twist coefficient is either 0 or $\frac{1}{n}$ for some integer $n$, $|n| \geq 2$ (see \cite{GabPFL}, \cite{KaRoFDTKTCT}). This restricts the total number of intersections that could occur, but even $\frac{1}{n}$-twisting could prevent the tunnel arc from being clean. Consider the geodesic representative of an arc $\alpha : [0, 1] \to F$, properly embedded, where $F$ has a hyperbolic structure (and totally geodesic boundary). If the endpoints $\alpha(0)$ and $\alpha(1)$ are too close together on $\bd F$, then the fractional Dehn twist could introduce intersections between the sub-arcs $h(\alpha([0, \epsilon]))$ and $\alpha([1-\epsilon, 1])$. However, if the endpoints $\alpha(0)$ and $\alpha(1)$ are evenly spaced on $\bd F$, then a fractional Dehn twist coefficient of $\frac{1}{n}$ for large enough $|n|$ would be sufficient to ensure cleanliness of the tunnel arc. 

Even with symmetric spacing, however, it is unclear whether a fractional Dehn twist coefficient of $\frac{1}{2}$ would introduce intersections between the arc and its image. In fact, Gabai conjectured that if a fibered knot in $S^3$ has a plumbed on Hopf band, then the Dehn twist coefficient of its monodromy is not $\frac{1}{2}$. Observe that this is an open conjecture, and if true, would imply that there exist fibered knots in $S^3$ whose monodromies have fractional Dehn twist coefficients $\frac{1}{2}$, but whose fibers contain no clean arcs at all, since an arc that is clean and alternating corresponds to a plumbed on Hopf band, and an arc that is clean and non-alternating implies that the fractional Dehn twist coefficient is 0 (see \cite{KaRoFDTKTCT}).

If Gabai's conjecture is true, then our prediction that tunnel number one, fibered knots always have clean tunnels would imply that such knots must have fractional Dehn twist coefficient of either 0 or $\frac{1}{n}$ for some integer $n$, $|n| > 2$.


\section{An Application to Hyperbolic Cusps}\label{section:cusps}

In \cite{FuScCGFM}, Futer and Schleimer study the hyperbolic structure on a hyperbolic surface bundle $M$. 
Each boundary component of $M$ is a cusp in the hyperbolic structure.
If we pick one boundary component, expanding a regular neighborhood of the corresponding cusp until it `bumps into itself' gives a well-defined `maximal cusp'. The geometric properties of the bounding torus of this neighborhood are invariants of the manifold $M$.
Futer and Schleimer relate this geometry to the action of the (pseudo-Anosov) monodromy on the arc complex of the fiber surface.

Given a compact, connected surface $F$, the \emph{arc complex} $\mathcal{A}(F)$ is a simplicial complex. The vertices of the complex are isotopy classes of essential arcs properly embedded in $F$. Distinct vertices span a simplex exactly when the isotopy classes of arcs can be simultaneously realized disjointly in $F$. 
A homeomorphism $h$ of $F$ induces a homeomorphism $h_*$ of $\mathcal{A}(F)$. 
The \emph{translation distance} $\dist_{\mathcal{A}}(h)$ of $h$ is 
\[
\dist_{\mathcal{A}}(h) = \min_{v \in \mathcal{A}^{(0)}(F)} \dist(v, h_*(v)) .
\]
Here the distance $\dist$ is measured in the $1$--skeleton $\mathcal{A}^{(1)}(F)$, where each edge has length $1$.
The \emph{stable translation distance} $\bar{\dist}_{\mathcal{A}}(h)$ is given by
\[
\bar{\dist}_{\mathcal{A}}(h) = \lim_{n \to \infty} \frac{\dist(v, h_*^n(v))}{n} ,
\]
where $v$ is any vertex of $\mathcal{A}(F)$.
The triangle inequality implies that $\bar{\dist}_{\mathcal{A}}(h) \leq \dist_{\mathcal{A}}(h)$.

We claim that a pseudo-Anosov homeomorphism cannot fix an essential arc in the surface. Assume $F$ is not an annulus.
Let $\gamma$ be an essential arc in $F$. Suppose that $h'\colon F\to F$ is a map isotopic to $h$ with $h'(\gamma)=\gamma$.
First assume that $\gamma$ has its endpoints on the same component of $\bd F$.
The endpoints of $\gamma$ divide the boundary component of $F$ into two arcs. Let $\gamma_1,\gamma_2$ be the essential simple closed curves given by combining each of these two arcs with a copy of $\gamma$. Then $h'$ fixes the multi-curve $\gamma_1\cup\gamma_2$.
On the other hand, assume $\gamma$ has its endpoints on distinct components of $\bd F$. 
Let $\gamma'$ be a simple closed curve that runs parallel to $\gamma$, around one boundary component of $\bd F$ on which $\gamma$ has an endpoint, back parallel to $\gamma$ and around the other boundary component. Then, up to isotopy, $h'(\gamma')=\gamma'$.
Thus, since a pseudo-Anosov homeomorphism cannot fix an essential multi-curve, it cannot fix an essential arc.

Written in this language, Proposition \ref{Proposition:TunnelCleanAfterIsotopy} says the following.

\begin{Corollary}\label{Corollary:TranslationDistanceIsOne}
Let $F$, $h$ and $M$ be as in Proposition \ref{Proposition:TunnelCleanAfterIsotopy}. 
Then $\dist_{\mathcal{A}}(h)\leq 1$.
If $h$ is pseudo-Anosov then $\dist_{\mathcal{A}}(h)=1$.
\end{Corollary}

Given this, \cite{FuScCGFM} Theorem 1.5 yields the following result.

\begin{Theorem}\label{Theorem:CuspAreaHeightBound}
Let $F$, $h$ and $M$ be as in Proposition \ref{Proposition:TunnelCleanAfterIsotopy}.
Suppose that $|\bd F|=1$ and $h$ is pseudo-Anosov.
Then the area of the maximal cusp is bounded above by $9 \chi(F)^2$, and the height of the cusp is strictly less than $-3 \chi(F)$.
\end{Theorem}
Here the \emph{height} of the cusp torus is its area divided by the length of the longitude.

We remark that \cite{FuScCGFM} Theorem 1.5 also gives lower bounds on these quantities in terms of $\bar{\dist}_{\mathcal{A}}(h)$. In \cite{GaTsMPTLCC}, Gadre and Tsai study the analogous distance in the curve complex, giving an explicit lower bound. It seems plausible that such a bound could likewise be obtained for the arc complex.

David Futer pointed out to the authors the following corollary of Corollary \ref{Corollary:TranslationDistanceIsOne}.

\begin{Corollary}\label{Corollary:AreaGrowsWithGenus}
There exists a family of fibered knots $K_n$, each having monodromy with translation distance 1,
such that the cusp area grows linearly with the knot genus.
\end{Corollary}

\begin{proof}
For $n\geq 1$, let $K_n$ be the $(6n+1)-$crossing knot with diagram $D_n$ formed from the blocks in Figure \ref{Figure:RationalKnotsSurfaces}, taking one of each of the outer two blocks and $n$ of the inner one. 
\begin{figure}[hbtp]
\begin{center}
\includegraphics[width=5in]{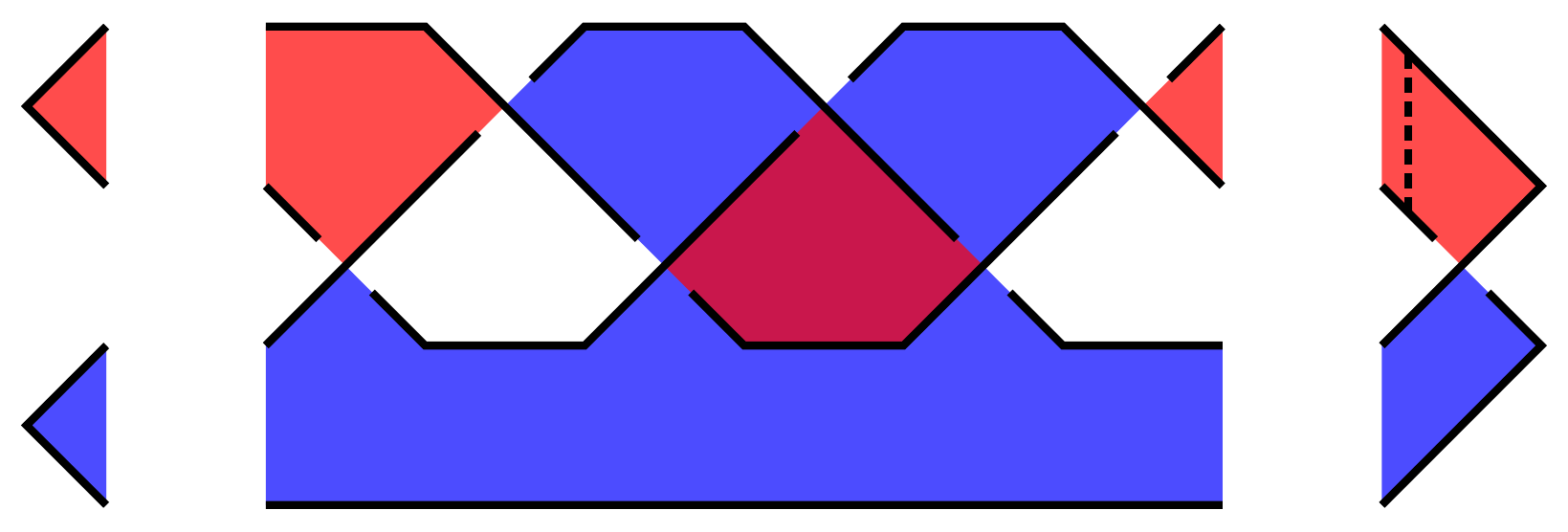}
\caption{We build the knot $K_n$ by combining $n$ copies of the middle block with one copy of each of the outer blocks.}
\label{Figure:RationalKnotsSurfaces}
\end{center}
\end{figure}  
In addition, let $R_n$ be the Seifert surface for $K_n$ constructed by combining the pieces of surface shown in Figure \ref{Figure:RationalKnotsSurfaces}. As $D_n$ is alternating, this surface has minimal genus.
Note that $\chi(R_n)=1-4n$, so $K_n$ has genus $2n$.

For $m\in\mathbb{N}$, let $f_m$ denote the $m^{th}$ term of the Fibonacci sequence (so $f_1=f_2=1$, $f_3=2$, $f_4 = 3$, $f_5=5$, etc.).
Then $K_n$ is the rational knot corresponding to the fraction $f_{6n+1}/f_{6n+2}$.
A rational knot with fraction $1/q$ for some $q$ is a torus knot, and all other rational knots are hyperbolic (see, for example, \cite{BrWuCOEDSBK}).
Two fractions $p_1/q_1$ and $p_2/q_2$ (with $p_i$ coprime to $q_i$) correspond to the same rational knot if and only if
$p_1=p_2$ and either $q_1\cong q_2 \mod p_1$ or $q_1q_2 \cong 1\mod p_1$.
Since $f_{6n+1}\neq 1$ for $n\geq 1$, this shows that $K_n$ is hyperbolic for each $n$.

That $R_n$ is a fiber surface can be checked directly by product disk decompositions (see \cite{GabDFLS3}) --- $2n$ product disk decompositions can be used to remove the `trefoil pattern' in the center of each of the $n$ middle blocks, leaving a checkerboard surface; further product decompositions can be used to reduce the surface further to a disk (by removing the white bigons in the remaining diagram).

Being rational knots, each $K_n$ has tunnel number one, with a tunnel given by the dotted arc in Figure \ref{Figure:RationalKnotsSurfaces}.
Therefore Corollary \ref{Corollary:TranslationDistanceIsOne} applies, and the monodromy of $K_n$ has translation distance $1$.

In a link diagram, a \emph{twist region} is a maximal collection of crossings connected in a line by bigons.
Each diagram $D_n$ is twist-reduced, and has $6n-1$ twist regions.
Thus \cite{FutKalPurCAFMAKT} Theorem 4.8 gives that, for the knot $K_n$, the area $a_n$ of the maximal cusp satisfies
\[\frac{1}{12}(6n-2)\leq a_n< \frac{40}{3}(6n-2). \qedhere \]
\end{proof}

Corollary \ref{Corollary:AreaGrowsWithGenus} shows that the dependence on Euler characteristic in the area bound in \cite{FuScCGFM} Theorem 1.5 and in Theorem \ref{Theorem:CuspAreaHeightBound} is necessary.


\bibliographystyle{hplain}   
\bibliography{MonodromyReferences}


\bigskip
 
\begin{footnotesize}
\noindent \textsc{CIRGET, UQAM, CP8888, Montr\'eal, H3C 3P8}\hfill

\noindent\textit{jessica.banks[at]lmh.oxon.org}\hfill
\\

\noindent \textsc{Imperial College London, South Kensington Campus, SW7 2AZ}\hfill

\noindent\textit{m.rathbun[at]imperial.ac.uk}\hfill
\end{footnotesize}

\end{document}